\documentclass[11pt]{article}\textwidth 160mm\textheight 235mm
\usepackage{amsfonts}
\usepackage{amssymb}
\usepackage{graphicx}
\usepackage{amsmath}
\usepackage{amsthm}
\usepackage{enumitem}
\usepackage{tikz}
\usepackage{mathrsfs}
\usepackage{cancel}
\usepackage{todonotes}
\usetikzlibrary{matrix}
\usetikzlibrary{arrows,shapes}
\usepackage{cite}
\usepackage{hyperref}
\hypersetup{colorlinks=true, urlcolor=blue}
\expandafter\let\expandafter\oldproof\csname\string\proof\endcsname
\let\oldendproof\endproof
\renewenvironment{proof}[1][\proofname]{%
  \oldproof[\ttfamily \scshape \bf #1. ]%
}{\oldendproof}
\def\O{{\cal O}}
\def\S{{\mathbb{S}}}

\def\B{\mathbb{B}}
\def\R{{\rm I\!R}}
\def\N{{\rm I\!N}}
\def\oi{{i_0}}
\def\ox{\bar{x}}
\def\oy{\bar{y}}
\def\oz{\bar{z}}
\def\ov{\bar{v}}

\def\ss{\scriptsize }

\def\ve{\varepsilon}

\def\hat{\widehat}
\def\tilde{\widetilde}
\def\emp{\emptyset}
\def\dist{{\rm dist}}
\def\span{{\rm span}\,}
\def\rge{{\rm rge\,}}

\def\Lm{{\Lambda}}
\def\tto{\rightrightarrows}

\def\d{{\rm d}}
\def\sub{\partial}

\def\Bar{\overline}
\def\ra{\rangle}
\def\la{\langle}
\def\ve{\varepsilon}

\def\cone{\mbox{\rm cone}\,}
\def\span{\mbox{\rm span}\,}

\def\inte{\mbox{\rm int}\,}
\def\gph{\mbox{\rm gph}\,}
\def\epi{\mbox{\rm epi}\,}
\DeclareMathOperator*{\mini}{minimize\;}
\DeclareMathOperator*{\argmin}{argmin}

\def\dom{\mbox{\rm dom}\,}

\def\ker{\mbox{\rm ker}\,}

\def\dn{\downarrow}
\def\O{\Omega}

\def\ph{\varphi}

\def\emp{\emptyset}

\def\oR{\Bar{\R}}
\def\lm{\lambda}
\def\olm{\bar\lambda}

\def\gg{\gamma}
\def\dd{\delta}
\def\al{\alpha}
\def\Th{\Theta}

\def\sm{\hbox{${1\over 2}$}}
\def\rsm{\hbox{${1\over 2r}$}}

\def \b{{\}_{k\in\N}}}
\def \c{{\}_{k\ge 0}}}
\def\L{{\mathscr{L}}}
\def\D{{\mathscr{D}}}
\def\sce{\setcounter{equation}{0}}

\def\menv{e_{1^{}/{\rho}}g} 
\def\menvk{e_{1^{}/{\rho_k}}g}
\def\prox{{\rm{prox}}_{\rho^{-1}{g}}} 
\def\proxk{{\rm{prox}}_{\rho_k^{-1}g}}

\def\x1k{{x^{k-1}}}
\def\xk{{x^k}}
\def\xkk{{x^{k+1}}}
\def\l1k{{\lambda^{k-1}}}
\def\lk{{\lambda^k}}
\def\lkk{{\lambda^{k+1}}}
\def\sig1k{{\sigma_{k-1}}}

\def\ro1k{{\rho_{k-1}}}
\def\rok{{\rho_k}}
\def\rokk{{\rho_{k+1}}}
\def\xik{{\beta_k}}
\def\e1k{{\varepsilon_{k-1}}}
\def\ek{{\varepsilon_k}}

\def\s1k{{z^{k-1}}}

\def\skk{{z^{k+1}}}

\begin{document}
\vspace*{0.5in}
\begin{center}
{\bf LOCAL CONVERGENCE ANALYSIS OF AUGMENTED LAGRANGIAN METHODS FOR PIECEWISE LINEAR-QUADRATIC COMPOSITE OPTIMIZATION PROBLEMS  }\\[1 ex]
NGUYEN T. V. HANG\footnote{Institute of Mathematics, Vietnam Academy of Science and Technology, Hanoi 10307, Vietnam (ntvhang@math.ac.vn).}
and  M. EBRAHIM SARABI\footnote{Department of Mathematics, Miami University, Oxford, OH 45065, USA (sarabim@miamioh.edu).}
\end{center}
\vspace*{0.05in}
\small{\bf Abstract.}  Second-order sufficient conditions for local optimality have been playing an important role in local convergence analysis 
of optimization algorithms. In this paper, we demonstrate that  this condition alone suffices to justify  the linear convergence of the primal-dual 
sequence, generated by the augmented Lagrangian method for piecewise linear-quadratic composite optimization problems, even when the Lagrange multiplier  in this class of problems is not unique. Furthermore, 
we establish the equivalence between the second-order sufficient condition  and the quadratic growth condition of the 
augmented Lagrangian problem for this class of composite optimization problems.
 \\[1ex]
\newtheorem{Theorem}{Theorem}[section]
\newtheorem{Proposition}[Theorem]{Proposition}
\newtheorem{Remark}[Theorem]{Remark}
\newtheorem{Lemma}[Theorem]{Lemma}
\newtheorem{Corollary}[Theorem]{Corollary}
\newtheorem{Definition}[Theorem]{Definition}
\newtheorem{Example}[Theorem]{Example}
\newtheorem{Algorithm}[Theorem]{Algorithm}
\renewcommand{\theequation}{{\thesection}.\arabic{equation}}
\renewcommand{\thefootnote}{\fnsymbol{footnote}}

\normalsize

\section{Introduction}\sce
One of the popular techniques for designing an efficient algorithm for solving  constrained and composite optimization has been to 
utilize a {\em regularization} for nonsmooth parts of the latter problems. Augmented Lagrangian methods (ALM, for short), which was
 first proposed independently by Hestenes and Powell for nonlinear programming problems  with equality constraints  in \cite{Hes69} and \cite{Pow69}, respectively, 
 can be considered as one of such methods that utilize  the Moreau envelope to deal with nonsmoothness in optimization problems. 
 The principal idea of the ALM   is to solve a {\em sequence} of  {\em subproblems} whose objective functions are  the {\em augmented Lagrangians}
 associated with a given optimization problem. The concept of the augmented Lagrangian for an extended-real-valued function,  introduced  by Rockafellar in \cite{Roc74} (see also  \cite[page~518]{rw}) via a perturbation scheme, 
   allows us to extend the ALM for different classes of constrained and composite optimization problems, see \cite{Roc73a,Roc73b,Roc74} for more details. 

 The main objective of this paper is to pursue  a local convergence analysis of the ALM for a large class of composite optimization problems whose  Lagrange multipliers are {\em not} necessary unique. 
Given twice continuously differentiable functions   $\ph:\R^n\to \R$ and  $\Phi:\R^n\to \R^m$,  a convex polyhedral set  $\Th$ in $\R^n$, and a convex piecewise linear-quadratic (CPLQ) function $g:\R^m\to \oR$,
 we aim to present a local convergence analysis of the ALM for  the composite optimization problem 
\begin{equation}\label{comp}
\mbox{minimize}\;\;\ph(x)+g\big(\Phi(x)\big)\quad \mbox{subject to}\;\; x\in \Th.
\end{equation} 
The CPLQ function $g$ in \eqref{comp} gives significant flexibility to  the composite  problem \eqref{comp} to cover important classes of optimization problems including classical 
nonlinear programming problems, constrained and unconstrained min-max optimization problems, and extended nonlinear programming problems introduced in \cite{r97}. 
For example, when the convex function $g$ is $\|\cdot\|_1$ or the pointwise maximum, the composite problem \eqref{comp} covers the regularized least square problem, known as {\em lasso}, 
and   constrained minimax problems, respectively; see Section~\ref{sect06} for a detailed discussion on important classes of the composite problem \eqref{comp}.
The polyhedral convex set $\Th$ in \eqref{comp} makes it possible to  cover nonnegativity constraints, upper and lower bounds on variables, and also situations where 
we want to minimize a function over a linear subspace or an affine subset of $\R^n$. 

The augmented Lagrangian $\L:\R^n\times \R^m\times (0,\infty)\to \R$ associated with the composite problem \eqref{comp} is defined  by 
\begin{equation}\label{aug}
\L(x,\lm,\rho):=\ph(x)+e_{ {1}/{\rho}} g \big(\Phi(x)+\rho^{-1}\lm\big)-\sm\rho^{-1}\|\lm\|^2
\end{equation}
with $e_{1/\rho} g$ signifying  the Moreau envelope of   $g$, given  by 
\begin{equation*}\label{Menv}
e_{1/\rho} g(y):=\inf_{z\in \R^m}\big\{g(z)+\frac{\rho}{2}\|y-z\|^2\big\}, \quad y\in \R^m.
\end{equation*}
It is well-known that the augmented Lagrangian $\L$ is continuously differentiable with respect to both $x$ and $\lm$ and is nondecreasing with the respect to $\rho$, the {\em penalty} parameter of the augmented Lagrangian. 
At  each iteration for a given pair $(\lm,\rho)\in \R^m\times (0,\infty)$,  our proposed  ALM for  the  composite problem \eqref{comp}
requires solving the {\em partially} augmented problem 
\begin{equation}\label{augp}
\mini \L(x, \lm, \rho)\quad \mbox{subject to}\;\; x\in \Th.
\end{equation} 
It is important to mention that the {\em full} augmentation of the composite problem \eqref{comp} is a special case of this  partially augmentation since the composite 
problem \eqref{comp} can be equivalently formulated as \eqref{comp} with $\Th=\R^n$. 

The local convergence analysis of the ALM for nonlinear programming problems with equality constraints  began in \cite{Hes69, Pow69}
and was extended to nonlinear programs with inequality constraints in \cite{Ber82}.  These results were achieved under rather restrictive assumptions such as  the  linear independence constraint qualification and the strict complementarity condition. 
Similar results were established for second-order cone programming and semidefinite programming problems in \cite{lz} and \cite{ssz}, respectively, under the counterparts of the latter conditions for these classes of problems.
These results were improved recently for ${\cal C}^2$-cone reducible constrained optimization problems in  \cite{KaS17,KaS19} (see also \cite{hms20}), where the linear convergence of the   ALM 
was justified under the second-order sufficient condition  and the strong Robinson constraint qualification. The common feature of the aforementioned results is that they require the uniqueness 
of Lagrange multipliers in  problems under consideration. A remarkable progress  was achieved by Fern\'andez and Solodov in \cite{fs12} regarding the local convergence analysis of the ALM for  nonlinear programming problems
with {\em nonunique} Lagrange multipliers, where they showed that the classical second-order sufficient condition {\em alone} suffices  to ensure the Q-linear (linear, for short) convergence of the primal-dual sequence, generated by this method;
see also \cite{iks15} for an improvement of this result. One interesting aspect of this achievement was to demonstrate that no constraint qualification is required for a local convergence 
analysis of the ALM for nonlinear programming problems. 

In this paper, we are going to demonstrate that a similar local convergence analysis and the rate of convergence of the primal-dual sequence of the ALM can be established for the 
composite problem \eqref{comp} under a counterpart of the second-order sufficient condition for this class of problems, formulated in \eqref{sosc}, without assuming  any 
constraint qualifications. This will be achieved by appealing to the concept of the second subderivative, introduced by Rockafellar in \cite{r88}, and to the theory of twice epi-differentiability 
of extended-real-valued functions, established recently in   \cite{mmsmor, mmstams, ms20}, which were not utilized in \cite{fs12}. These tools of second-order generalized differentiation allow us to obtain a second-order characterization of the quadratic growth condition for the partially augmented problem \eqref{augp}, which was known to play an indispensable role in the convergence analysis of optimization algorithms. 
Using the latter quadratic growth condition, we show that the partially augmented problem \eqref{augp} is always solvable and its optimal solution mapping is  uniformly  isolated calm
provided that the penalty parameter $\rho$ therein is sufficiently large. 
Appealing to an error bound estimate for the KKT system of \eqref{comp}, established recently in \cite{s20},   to an error bound for consecutive iterates of the ALM, obtained in Theorem~\ref{est}, 
and employing an iterative framework, proposed by Fisher in \cite[Theorem~1]{Fis02}, to deal with generalized equations with non-isolated solutions, we justify the linear convergence 
of the primal-dual sequence of the ALM for the composite problem \eqref{comp} when the penalty parameter $\rho$ is sufficiently large. 

The rest of the paper is organized as follows. Section ~\ref{sect02} contains  tools of variational analysis, utilized throughout the paper. 
Section~\ref{sect03} provides important first- and second-order variational properties of the augmented Lagrangian \eqref{aug}. In particular, we 
obtain a characterization of the quadratic growth condition for \eqref{augp} via the second subderivative. Section~\ref{sect04} concerns a characterization 
of the quadratic growth condition for the partially augmented problem \eqref{augp} via the second-order sufficient condition. Furthermore, we justify 
a uniform version of the quadratic growth condition for \eqref{augp}, important for our local convergence analysis. Finally, Section~\ref{sect05} contains 
a local convergence analysis and the rate of convergence of the primal-dual sequence constructed  by the ALM for \eqref{comp}.

\section{ Preliminary Definitions and Results}\sce \label{sect02}

In this section, we first briefly review basic constructions of variational analysis and generalized differentiation employed in the paper. In what follows,    we 
 denote by $\B$ the closed unit ball in the space in question and by $\B_r(x):=x+r\B$ the closed ball centered at $x$ with radius $r>0$. 
 In the  product space $\R^n\times \R^m$, we use the norm $\|(w,u)\|=\sqrt{\|w\|^2+\|u\|^2}$ for any $(w,u)\in \R^n\times \R^m$.
 Given a nonempty set $C\subset\R^n$, the symbols $\inte C$,  $\cone C$, and $\span C$ signify its interior, conic hull, and the   linear space generated by  $C$, respectively. 
 For any set $C$ in $\R^n$, its indicator function is defined by $\dd_C(x)=0$ for $x\in C$ and $\dd_C(x)=\infty$ otherwise. We denote
 by $P_C$ the projection mapping onto $C$ and  by $\dist(x,C)$  the distance between $x\in \R^n$ and a set $C$.
 For a vector $w\in \R^n$, the subspace $\{tw |\, t\in \R\}$ is denoted by $[w]$. 
We write $x(t)=o(t)$ with $x(t)\in \R^n$ and $t>0$ to mean that ${\|x(t)\|}/{t}$ goes to $0$ as $t\dn 0$.
Finally, we denote by $\R_+$ (respectively,  $\R_-$) the set of non-negative (respectively, non-positive) real numbers.
For an $n\times m$ matrix $A$, we denote by $\ker A$ and $\rge A$ its null and range spaces, respectively.

A sequence    $\{C^k\b$ of subsets in $\R^n$ is called convergent to a set $C\subset\R^n$ as $k\to \infty$ if $C$ is closed and
\begin{equation*}
\lim_{k\to \infty}{\rm dist}(w,C^k)={\rm dist}(w,C)\;\mbox{ for all }\;w\in\R^n.
\end{equation*}
Given a nonempty set $C\subset\R^n$ with $\ox\in C$, the  tangent cone $T_ C(\ox)$ to $C$ at $\ox$ is defined by
\begin{equation*}\label{2.5}
T_C(\ox)=\big\{w\in\R^n|\;\exists\,t_k{\dn}0,\;\;w_k\to w\;\;\mbox{ as }\;k\to\infty\;\;\mbox{with}\;\;\ox+t_kw_k\in C\big\}.
\end{equation*}
When $C$ is convex, the  {normal cone} to $C$ at $\ox\in C$ in the sense of convex analysis  is given by
\begin{equation*}
N_C(\ox)=\big\{v\in\R^n\;\big|\; \la v,x-\ox\ra\le 0\quad \mbox{for all}\;\; x\in C\big\}.
\end{equation*}
For a convex  function $f:\R^n \to \oR:=(-\infty,\infty]$  with $\ox\in \dom f$, its   subdifferential at $\ox$   is defined  by 
\begin{equation*}\label{sub}
\partial f(\ox)=\big\{v\in\R^n\;\big|\;  \la v,x-\ox\ra\le f(x)-f(\ox)\;\;\mbox{for all}\;x\in \R^n\big\}.
\end{equation*}

A sequence $\{f^k\b$ of functions $f^k:\R^n\to \oR$ is said to {\em epi-converge} to a function $f:\R^n\to [-\infty,\infty]$ if we have $\epi f^k\to \epi f$ as $k\to \infty$.
We denote by $f^k\overset {\ss \mbox{e}}{\to} f$ the  epi-convergence of the sequence $\{f^k\b$ to $f$.
Given  a function  $f:\R^n \to \oR$ and  a point $\ox$ with $f(\ox)$ finite, the subderivative function $\d f(\ox)\colon\R^n\to[-\infty,\infty]$ is defined by
\begin{equation}\label{fsud}
{\mathrm d}f(\ox)(w)=\liminf_{\substack{
   t\dn 0 \\
  w'\to w
  }} {\frac{f(\ox+tw')-f(\ox)}{t}}.
\end{equation}
 The critical cone of $f$ at $\ox$ for $\ov$ with $(\ox,\ov)\in \gph \sub f$ is defined by 
$$
{K_f}(\ox,\ov)=\big\{w\in \R^n\big|\;\la\bar v,w\ra=\d f(\ox)(w)\big\}.
$$
When $f=\dd_C$, namely the indicator function of $C$, the critical cone of $\dd_C$ at $\ox$ for $\ov$ is denoted by $K_C(\ox,\ov)$. In this case, the above definition of the critical cone of a function 
boils down to  the well-known concept of the critical cone of a set (see \cite[page~109]{dr}), namely $K_C(\ox,\ov)=T_C(\ox)\cap [\ov]^\perp$.
Define the parametric  family of 
second-order difference quotients for $f$ at $\ox$ for $\ov\in \R^n$ by 
\begin{equation*}\label{lk01}
\Delta_t^2 f(\bar x , \ov)(w)=\dfrac{f(\ox+tw)-f(\ox)-t\langle \ov,\,w\rangle}{\frac{1}{2}t^2}\quad\quad\mbox{with}\;\;w\in \R^n, \;\;t>0.
\end{equation*}
If $f(\ox)$ is finite, then the {second subderivative} of $f$ at $\ox$ for $\ov$   is given by 
\begin{equation*}\label{ssd}
\d^2 f(\bar x , \ov)(w)= \liminf_{\substack{
   t\dn 0 \\
  w'\to w
  }} \Delta_t^2 f(\ox , \ov)(w'),\;\; w\in \R^n.
\end{equation*}
The second subderivative of $f$ at $\ox$ (without mention of $\ov$) is defined by 
\begin{equation}\label{ssud}
\d^2 f(\ox)(w)=\liminf_{\substack{t\dn 0\\
w'\to w}} \frac{f(\ox+tw')-f(\ox)-t\d f(\ox)(w')}{\sm t^2},\quad w\in \R^n.
\end{equation}
Following \cite[Definition~13.6]{rw}, a function $f:\R^n \to \oR$ is said to be {twice epi-differentiable} at $\bar x$ for $\ov\in\R^n$, with $f(\ox) $ finite, 
if the sets $\epi \Delta_t^2 f(\bar x , \ov)$ converge to $\epi \d^2 f(\bar x,\ov)$ as $t\downarrow 0$. 
In what follows, we refer to the second subderivative of $f$ at $\bar x$ for $\ov$ by the {\em second-order epi-derivative} of $f$ when it is twice epi-differentiable at   $\bar x$ for $\ov$.
 
Recall  that a function $g:\R^m\to \oR$ is called  piecewise linear-quadratic if $\dom g = \cup_{i=1}^{s} C_i$ with $s\in \N$ and $C_i $ being polyhedral convex  sets for $i = 1, \ldots, s$, and if $g$ has a representation of the form
\begin{equation} \label{PWLQ}
g(z) = \sm \langle A_i z ,z \rangle + \langle a_i ,z \rangle + \alpha_i  \quad \mbox{for all} \quad  z \in C_i,
\end{equation}
where $A_i$ is an $m \times m$ symmetric matrix, $a_i\in \R^m$, and $\alpha_i\in \R$ for all $i = 1, \cdots, s$. Take $\oz\in \dom g$ and define the {\em active} indices  of the domain of $g$ at $\oz$ by
\begin{equation}\label{act}
I(\oz)=\big \{i\in \{1,\ldots, m\} |\, \oz\in C_i\big \}.
\end{equation}
If in addition $g$ is convex, 
it follows   from \cite[page~487]{rw} that 
if  $\oz\in \dom g$, the subdifferential of $g$ at  $\oz$ can be calculated by 
\begin{equation}\label{sub}
\sub g(\oz)=\bigcap_{i\in I(\oz)} \big \{v\in \R^m |\, v-A_i \oz- a_i\in N_{ C_i}(\oz)\big \}.
\end{equation}

Next, we recall some second-order variational properties of CPLQ functions, which are  taken from  \cite[Theorem~13.14]{rw} and \cite[Proposition~13.9]{rw}. 

\begin{Proposition}[second-order variational properties of CPLQ] \label{sop} Assume that $g:\R^m\to \oR$ is a CPLQ function with representation \eqref{PWLQ}
and that $\oz\in \dom g$ and $\ov\in \sub g(\oz)$. Set $\ov_i:=\ov -A_i \oz -a_i$ for all $i\in I(\oz)$. 
Then the following conditions hold:
 \begin{itemize}[noitemsep,topsep=0pt]
 \item [{\rm (a)}]   the critical cone of $g$ at $\oz$ for $\ov$ can be calculated by $K_g(\oz,\ov)=N_{\sub g(\oz)}(\ov)$. Moreover, we have 
\begin{equation}\label{cc2}
  \d^2 g(\oz , \ov ) (w)=\d^2 g(\oz)(w)+\dd_{K_g(\oz,\ov)}(w), \quad w\in \R^m.
\end{equation}
 \item [{\rm (b)}] the function $g$ is twice epi-differentiable at $\ox$ for $\ov$ and its second-order epi-derivative  at this point can be calculated by 
\begin{equation}\label{pwfor}
 \d^2 g(\oz , \ov ) (w)  
=\begin{cases}
\la A_i w, w\ra&\mbox{if}\;\; w\in {K}_{C_i}(\oz,\ov_i),\\
\infty&\mbox{otherwise}.
\end{cases}
\end{equation}
Moreover, we have 
\begin{equation}\label{domcr}
K_g(\oz,\ov)=\bigcup_{i\in I(\oz)} {K}_{C_i}(\oz,\ov_i).
\end{equation}
  \end{itemize}
\end{Proposition}

The   Karush-Kuhn-Tucker (KKT)   system associated with the composite problem \eqref{comp} is given by 
\begin{equation}\label{vs} 
0\in \nabla_x L(x,\lm)+N_\Th(x),\;\;\lm\in \sub g(\Phi(x)),
\end{equation}
where $L(x,\lm)= \ph(x)+\la \Phi(x),\lm\ra$ for any   $(x,\lm)\in \R^n\times \R^m$. 
Given a point $\ox\in\R^n$, we define the set of {  Lagrange multipliers} of the KKT system \eqref{vs} associated with $\ox$ by
\begin{equation}\label{laset}
\Lambda(\ox):=\big\{\lm\in\R^m\;\big|\; 0\in \nabla_x L(\ox,\lm)+N_\Th(\ox),\;\lm\in \sub g(\Phi(\ox))\big\}.
\end{equation}
If  $(\ox,\olm)$ is a solution to the KKT  system \eqref{vs},   we clearly have   $\olm\in\Lambda(\ox)$.
 Given  a solution $(\ox,\olm)$  to \eqref{vs}, the second-order sufficient condition (SOSC)
for the composite problem \eqref{comp} at $(\ox,\olm)$ is formulated by 
\begin{equation}\label{sosc}
\langle\nabla_{xx}^2L(\bar x,\olm)w,w\rangle+\d^2g (\Phi(\bar x),\olm) (\nabla \Phi(\ox)w )>0\quad \mbox{for all }\; w\in \D\setminus\{0\},
\end{equation}
where the convex cone $\D$ is defined by 
 \begin{equation}\label{coned}
 \D:=K_\Th\big(\ox,-\nabla_x L(\ox,\olm)\big)\cap \big\{w\in \R^n\big|\; \nabla \Phi(\ox)w\in K_g \big(\Phi(\ox),\olm \big)\big\}.
 \end{equation}
When $g$ is the  indicator function of a polyhedral convex set and $\Th=\R^n$,
one can observe that  this SOSC boils down to that of  classical nonlinear programming problems. 
To demonstrate the importance of this condition in the local convergence analysis of the ALM for \eqref{comp}, we recall below  
a result, obtained recently in \cite{s20}, showing that it yields an error bound estimate for the KKT system \eqref{vs}, which is central to our   local convergence analysis  in this paper. To this end,  recall that, given a constant $r>0$,    the proximal mapping of  a convex function $g:\R^m\to \oR$ is defined by 
\begin{equation*}\label{pr}
{\rm{prox}}_{rg}(x):=\underset{z\in \R^n}{\argmin} \big\{g(z)+\rsm\|x-z\|^2\big\},\quad x\in\R^m.
\end{equation*}
In what follows, when $r=1$, the proximal mapping of $g$ will be denoted by ${\rm{prox}_g}$. For such a convex function, 
it is known that for any $x\in \R^m$ we have 
\begin{equation}\label{prm}
{\rm{prox}}_{rg}(x)=\big(I+r\sub g\big)^{-1}(x) \quad \mbox{and}\quad \nabla e_rg(x)=\big(rI+(\sub g)^{-1}\big)^{-1}(x),
\end{equation}
where $I$ stands for the $m\times m$ identity matrix; see \cite[Proposition~12.19]{rw} and \cite[Theorem~2.26]{rw} for more details. 
We are now in a position to state the aforementioned result for the KKT system \eqref{vs}, which comes from combining  \cite[Theorem~3.6]{s20} and  \cite[Propositions~3.8 and 3.9]{s20}.
 \begin{Theorem}[error bound via second-order optimality conditions]\label{sooc} 
 Assume that $(\ox,\olm)$ is a solution to the KKT system \eqref{vs}. If the SOSC \eqref{sosc} holds at $(\ox,\olm)$, then
 there are positive constants $\gg_1$ and $\kappa_1$ for which  the error bound estimate
\begin{equation}\label{subr3}
\|x-\ox\|+\dist\big(\lm,\Lambda(\ox)\big)\le\kappa_1\Big(\dist\big(-\nabla_x L(x,\lm),N_\Th(x)\big)+\| \Phi(x)-{\rm{prox}}_g\big(\lm+ \Phi(x)\big)\|\Big)
\end{equation}
holds for any $(x,\lm)\in\B_{\gg_1}(\ox,\olm)$.
 \end{Theorem}
 
 Note that the SOSC \eqref{sosc} is strictly stronger  than the error bound estimate \eqref{subr3}; see \cite{s20} for more details. Indeed, it 
 was shown in \cite[Theorem~3.6]{s20} that the latter error bound estimate is equivalent to the fact that the Lagrange multiplier $\olm$ therein is {\em noncritical}
 in sense of \cite[Definition~3.1]{s20}. While this may suggest that the SOSC \eqref{sosc} can be replaced with the latter noncriticality assumption to 
 ensure the error bound estimate \eqref{subr3}, it is not clear whether this can be done for all other results, established in this paper.

 \section{Variational Properties of Augmented Lagrangians}\sce\label{sect03}

This section aims to present important first- and second-order variational properties of the augmented Lagrangian \eqref{aug}
that play a fundamental role in the local convergence analysis of the ALM in this paper. We begin with recalling the behavior 
of the augmented Lagrangian \eqref{aug} with respect to the multiplier $\lm$ and the parameter $\rho$. We also 
show that the augmented Lagrangian of the composite problem \eqref{comp} is fully amenable in the sense of \cite[Definition~10.23]{rw}: 
A function $f:\R^n\to \oR$ is called fully amenable at $\ox\in \dom f$ if there is a neighborhood ${\cal O}$ of $\ox$ on which 
$f$ can be represented in the form $f=\psi\circ h$, where $\psi:\R^m\to \oR$ is a CPLQ function and   $h:\R^n\to \R^m$ is a ${\cal C}^2$ function, and the 
basic constraint qualification 
\begin{equation}\label{bcq}
N_{\ss \dom \psi}\big(h(\ox)\big)\cap \ker \nabla h(\ox)^*=\{0\}
\end{equation}
is satisfied. 

\begin{Proposition}[properties of the augmented Lagrangian]\label{paug}
Assume that  $(x, \lm, \rho)\in \R^n\times \R^m\times (0, \infty)$ and $\ox\in \R^n$. Then the following conditions hold for the  augmented Lagrangian   \eqref{aug}:
\begin{enumerate}[noitemsep,topsep=0pt]
\item the function $x \mapsto \L(x, \lm, \rho)$ is $\mathcal{C}^1$ and fully amenable at $\ox$;
\item the function $\lm \mapsto \L(x, \lm, \rho)$ is concave;
\item the function $\rho \mapsto \L(x, \lm, \rho)$ is nondecreasing.
\end{enumerate}
\end{Proposition} 

\begin{proof} Let $(x, \lm, \rho)\in \R^n\times \R^m\times (0, \infty)$ and $\ox\in \R^n$. To justify (a), it follows from \cite[Proposition~12.30]{rw} that the envelope function $\menv$ is CPLQ. 
Moreover, the envelope function $\menv$  is continuously differentiable on $\R^m$ and thus the function $x\mapsto e_{ {1}/{\rho}} g \big(\Phi(x)+\rho^{-1}\lm\big)$
is $\mathcal{C}^1$ and fully amenable at $\ox$. In this case, the basic qualification condition \eqref{bcq} holds automatically at any $\ox\in \R^n$ since $\dom \menv=\R^m$. Clearly the mapping $x\mapsto \ph(x)$
is fully amenable at $\ox$ since $\ph$ is ${\cal C}^2$. Combining these and remembering that the sum of two ${\cal C}^1$ and fully amenable functions is fully amenable (cf. \cite[Exercise~10.22]{rw})
indicate that the augmented Lagrangian \eqref{aug} is fully amenable at $\ox$. Parts (b) and (c) were established in \cite[Exercise~11.56]{rw}.
\end{proof} 

The next result shows that    KKT points of the composite problem \eqref{comp} are KKT points of the augmented problem \eqref{augp}.

\begin{Proposition}[propagation of stationary conditions]\label{fopag}
Let $(\ox,\olm)$ be a solution to the KKT system \eqref{vs}. Then for any $\rho>0$,  the following conditions are satisfied:
\begin{enumerate}[noitemsep,topsep=0pt]
\item we always have  $\L(\ox,\olm,\rho)=\ph(\ox)+g\big(\Phi(\ox)\big)$;
\item $\ox$ is a stationary point of the augmented  problem 
\begin{equation}\label{ap}
\mini\; \L(x,\olm,\rho)\quad \mbox{subject to}\;\; x\in \Th.
\end{equation}
\end{enumerate}
\end{Proposition}
\begin{proof} Given $\rho>0$, it follows from the first equality in \eqref{prm} and  $\olm\in \sub g\big(\Phi(\ox)\big)$ that 
 $$
 \prox\big(\Phi(\ox)+\rho^{-1}\olm\big)=\Phi(\ox),
 $$ which clearly proves (a).
We proceed now with the proof of  (b). Since $g$ is convex, we conclude from the second equality in \eqref{prm} that the condition $\olm\in \sub g\big(\Phi(\ox)\big)$ is equivalent to $\nabla(e_{ {1}/{\rho}} g)\big(\Phi(\ox)+\rho^{-1}\olm\big)=\olm$
for any $\rho>0$.  Thus we get 
\begin{equation}\label{kkt2}
\nabla_x \L(\ox,\olm,\rho)=\nabla \ph(\ox)+ \nabla \Phi(\ox)^*\nabla(e_{ {1}/{\rho}} g)(\Phi(\ox)+\rho^{-1}\olm)=\nabla_x L(\ox,\olm).
\end{equation}
Combining these and \eqref{vs} justifies  (b).
\end{proof}

The next result presents a characterization of the quadratic growth condition of the augmented problem \eqref{augp} via the second subderivative 
of the augmented Lagrangian. Such a characterization has important applications in the next section,  where we will show that  the quadratic growth condition 
for the augmented problem \eqref{augp}  amounts to  the SOSC \eqref{sosc}. 

\begin{Proposition}[quadratic growth condition for augmented Lagrangians]\label{2ndgr}
Let $(\ox,\olm)$ be a solution to the KKT system \eqref{vs} and set $\ov:= \nabla_x L(\ox,\olm)$. Then for any $\rho>0$,  the following conditions are equivalent: 
\begin{enumerate}[noitemsep,topsep=0pt]
\item    there exists a constant $\bar\rho>0$ such that for any $\rho\ge\bar\rho$ the second-order condition 
\begin{equation}\label{eq26}
\d^2_x\L\big((\ox,\olm,\rho),\ov\big)(w)>0\;\mbox{ for all }\;w\in K_\Th (\ox,-\ov ) \setminus\{0\}
\end{equation}
holds;
 \item  there exist positive constants $\bar\rho$, $\gamma$, and $\ell$ such that for any $\rho\ge\bar \rho$, the quadratic growth condition 
 \begin{equation}\label{eq16}
\L(x,\olm,\rho)\ge \ph(\ox)+g\big(\Phi(\ox)\big)+\ell\| x-\bar x\|^2\quad \mbox{ for all }\;x\in\B_\gg (\ox)\cap \Th
\end{equation}
holds. 
\end{enumerate}
\end{Proposition} 

\begin{proof} Pick $\rho>0$ and    $x\in \R^n$, and 
 define the function $  f_\rho(x) := \L(x,\olm, \rho)+\delta_\Th(x)$.  Remember that $(\ox,\olm)$ is a solution to the KKT system  \eqref{vs}.
Since $\partial   f_\rho(\ox) = \nabla_x \L(\ox,\olm,\rho)+N_\Th(\ox)$, it follows from \eqref{kkt2} that $0\in \sub f_\rho(\ox)$. 
Employing \cite[Proposition~13.19]{rw} and Proposition~\ref{paug}(a), we conclude  that  
\begin{eqnarray}
\d^2  f_\rho (\ox, 0)(w) &=&\d^2_x\L\big((\ox,\olm,\rho),\ov\big)(w)+\d^2\delta_\Th(\ox, -\ov)(w)\nonumber \\
&=&\d^2_x\L\big((\ox,\olm,\rho),\ov\big)(w)+\delta_{K_\Th(\ox, -\ov)}(w),\label{sumr}
\end{eqnarray} 
where the formula for the second subderivative of $\dd_\Th$ was taken from \cite[Example~3.4]{mmstams}.

After this preparation,  assume that   (a)  holds. So we find   $\bar \rho>0$ for which \eqref{eq26} is satisfied. 
This, combined with \eqref{sumr}, tells us that $\d^2  f_{\bar\rho}(\ox, 0)(w)>0$  for all $w\in \R^n\setminus\{0\}$.
Employing now   \cite[Theorem~13.24(c)]{rw}, we find positive constants  $\gg$ and $\ell$ for which the quadratic 
growth condition 
$$
\L(x,\olm,\bar \rho)\ge \L(\ox,\olm,\bar \rho) +\ell\| x-\bar x\|^2 =\ph(\ox)+g\big(\Phi(\ox)\big)+\ell\| x-\bar x\|^2
$$
holds for all $x\in\B_\gg (\ox)\cap \Th$, where the last equality comes from Proposition~\ref{fopag}(a). By Proposition~\ref{paug}(c), the mapping $\rho\mapsto \L(x,\olm,  \rho)$ is nondecereasing. Combining this with the above quadratic growth condition justifies \eqref{eq16} for any $\rho\ge \bar\rho$.

Turning now to the opposite implication, assume that (b) holds and then pick $\rho\in [\bar\rho, \infty)$. By  Proposition~\ref{fopag}(a), we have $f_\rho(\ox)= \ph(\ox)+g\big(\Phi(\ox)\big)$, which yields 
$$
f_\rho(x) \ge f_\rho(\ox)+\ell\| x-\bar x\|^2\quad \mbox{ for all }\;x\in\B_\gg (\ox).
$$
This implies via the definition of the second subderivative that $\d^2  f_{\rho}(\ox, 0)(w)>0$  for all $w\in \R^n\setminus\{0\}$. Appealing now to \eqref{sumr} verifies \eqref{eq26} and hence ends the proof.
\end{proof} 

We end this section by an efficient formula for the second-order epi-derivative of the augmented Lagrangian \eqref{aug}.
Recall that a function $f:\R^n\to \oR$ is called semidifferentiable at $\ox$ if it is finite at $\ox$ and the ``$\liminf$" in \eqref{fsud} is a ``$\lim$."
Furthermore, $f$ is said to be twice semidifferentiable at $\ox$ if it is semidifferentiable at $\ox$ and  the ``$\liminf$" in \eqref{ssud} is a ``$\lim$."
The second semiderivative of $f$ at $\ox$ will be denoted by $\d^2f(\ox)$. It is known that  the twice semidifferentiability  has serious limitations to handle nonsmoothness
in variational analysis; see   \cite[p.\ 590]{rw} for a detailed discussion. As argued below, this property can be established for the augmented Lagrangian \eqref{aug}, however. Such a result was first 
appeared in \cite[Theorem~8.3]{mmstams} for parabolically regular constrained problems at their KKT points. Below we show that a similar conclusion holds for the composite problem \eqref{comp}
without restricting it to the KKT points of the latter problem.

\begin{Theorem}[second subderivatives of augmented Lagrangians]\label{ssc} 
Assume that  $(x, \lm, \rho)\in \R^n\times \R^m\times (0, \infty)$ and $\ox\in \R^n$. 
Then  the function $x\mapsto\L(x,\lm,\rho)$, defined via the augmented Lagrangian \eqref{aug}, is twice epi-differentiable at $\ox$ for $\ov:=\nabla_x \L(\ox,\lm,\rho)$ and its second-order epi-derivative 
can be calculated for any $w\in \R^n$ by
\begin{equation}\label{eq40}
\d^2_x\L\big((\ox,\lm,\rho),\ov\big)(w)=\big\langle \nabla^2_{xx}L(\ox,\mu)w,w\big\rangle+ e_{1/{2\rho}}\big(\d^2 g\big(\Phi(\ox)+\rho^{-1}(\lm-\mu),\mu\big)\big) \big(\nabla \Phi(\ox)w \big), 
\end{equation}
where $\mu=\nabla_x \big(e_{1/{\rho}}g\big)\big(\Phi(\ox)+\rho^{-1}\lm\big)$.
\end{Theorem}
\begin{proof}  Given $\rho>0$,  it follows from  the second equality in \eqref{prm} that $\mu\in \sub g\big(\Phi(\ox)+\rho^{-1}(\lm-\mu)\big)$. 
According to Proposition~\ref{sop}(b), the CPLQ function $g$ is twice epi-differentiable at $\Phi(\ox)+\rho^{-1}(\lm-\mu)$ for $\mu$. Appealing now 
to  \cite[Proposition~4.1]{hms20} tells us that    the   envelope function $e_{1/\rho} g$ is twice epi-differentiable at $\Phi(\ox)+\rho^{-1}\lm$ for $\mu$ and that 
\begin{equation}\label{more}
\d^2 (e_{1/\rho} g )\big(\Phi(\ox)+\rho^{-1}\lm,\mu\big)=e_{1/2\rho}\big(\d^2 g(\Phi(\ox)+\rho^{-1}(\lm-\mu),\mu)\big).
\end{equation}
By \eqref{pwfor}, the second-order epi-derivative of $g$ is CPLQ. This, combined with \eqref{more}, indicates that $ e_{1/2\rho}\big(\d^2 g(\Phi(\ox)+\rho^{-1}(\lm-\mu),\mu)\big)$
is finite on $\R^m$. Employing now \cite[Theorem~4.3]{r20}, we conclude that  $e_{1/\rho}g$ is twice semidifferentiable at $\Phi(\ox)+\rho^{-1}\lm$ and its second semiderivative 
coincides with its second-order epi-derivative, namely 
\begin{equation}\label{ssed}
\d^2\big(e_{1/\rho}g\big)\big(\Phi(\ox)+\rho^{-1}\lm\big)=\d^2 (e_{1/\rho} g )\big(\Phi(\ox)+\rho^{-1}\lm,\mu\big).
\end{equation}
This, together with the chain and sum rules for twice semidifferentiability, established in  \cite[Propositions~8.2(i) and 8.1]{mmstams}, 
shows that the mapping $x\mapsto\L(x,\lm,\rho)$ is twice semidifferentiable at $\ox$. Thus, by   \cite[Propositions~8.2(iii)]{mmstams}, the latter mapping is twice epi-differentiable at $\ox$ for $\ov$ and 
\begin{equation}\label{lss}
\d_x^2{\L}(\ox,\lm,\rho)(w)=\d^2_x\L\big((\ox,\lm,\rho),\ov\big)(w), \quad w\in \R^n.
\end{equation}
By  \cite[Propositions~8.2(i) and 8.1]{mmstams},  the second  semiderivative of $x\mapsto\L(x,\lm,\rho)$ at $\ox$ can be calculated for any $w\in \R^n$ by
\begin{eqnarray*}
\d_x^2{\L}(\ox,\lm,\rho)(w)&=&\d^2\ph(\ox)(w)+\big\la\mu,\nabla^2\Phi(\ox)(w,w)\big\ra+\d^2\big(e_{1/\rho} g\big)\big(\Phi(\ox)+\rho^{-1}\lm\big)\big(\nabla \Phi(\ox)w\big)\\
&=&\la\nabla^2\ph(\ox)w,w\ra+\big\la\mu,\nabla^2\Phi(\ox)(w,w)\big\ra+\d^2\big(e_{1/\rho} g\big)\big(\Phi(\ox)+\rho^{-1}\lm,\mu\big)\big(\nabla \Phi(\ox)w\big)\\
&=&\big\la\nabla_{xx}^2L(\ox,\mu)w,w\big\ra+e_{1/2\rho}\big(\d^2 g(\Phi(\ox)+\rho^{-1}(\lm-\mu),\mu)\big)\big(\nabla \Phi(\ox)w\big),
\end{eqnarray*}
where the second equality comes from \eqref{ssed} and the last one results from \eqref{more}. This together with \eqref{lss}
justifies \eqref{eq40} and hence completes the proof.
\end{proof}

The second-order epi-derivative formula \eqref{eq40} of the augmented Lagrangian \eqref{aug} can be slightly simplified when  $(\ox, \lm)$ in Theorem~\ref{ssc}
is a solution to the KKT system \eqref{vs}, as shown below.
\begin{Corollary}[second-order epi-derivatives] Let $(\ox,\olm)$ be a solution to the KKT system \eqref{vs} and set $\ov:= \nabla_x L(\ox,\olm)$. Then for any $\rho>0$,
the function $x\mapsto\L(x,\lm,\rho)$, defined via the augmented Lagrangian \eqref{aug}, is twice epi-differentiable at $\ox$ for $\ov $ and its second-order epi-derivative 
can be calculated for any $w\in \R^n$ by
\begin{equation}\label{eq402}
\d^2_x\L\big((\ox,\olm,\rho),\ov\big)(w)=\big\langle \nabla^2_{xx}L(\ox,\olm)w,w\big\rangle+ e_{1/{2\rho}}\big(\d^2 g\big(\Phi(\ox),\olm\big)\big) \big(\nabla \Phi(\ox)w \big).
\end{equation}
\end{Corollary}
\begin{proof} Because  $(\ox,\olm)$ is a solution to the KKT system \eqref{vs}, we get $\olm\in \sub g\big(\Phi(\ox)\big)$.
Since $g$ is CPLQ, we conclude from the second equality in \eqref{prm} that the latter inclusion is equivalent to $\nabla(e_{ {1}/{\rho}} g)\big(\Phi(\ox)+\rho^{-1}\olm\big)=\olm$
for any $\rho>0$. This means that $\mu=\olm$, where $\mu$ was taken from Theorem~\ref{ssc}. Moreover, it follows from 
\eqref{kkt2} that $\nabla_x \L(\ox,\olm,\rho)=\nabla_x L(\ox,\olm)=\ov$.
Combining these and the second-order epi-derivative formula \eqref{eq40}, 
we arrive at \eqref{eq402}.
\end{proof}


\section{  Quadratic Growth Conditions for Augmented Lagrangians}\label{sect04} \sce

This section is devoted to the study  of the quadratic growth condition of the augmented Lagrangian \eqref{aug}.
In fact, we are going to show that such a growth condition for \eqref{aug} amounts to the validity of the SOSC \eqref{sosc} for the  composite problem \eqref{comp}.
This characterization  plays a major role in our local convergence analysis of the ALM for \eqref{comp} in the next section. The proof is heavily 
relying on the powerful theory of epi-convergence of functions as well as the characterization of the quadratic growth condition via the second subderivative.

\begin{Theorem}[characterizations of quadratic growth condition]\label{growth}
Let $(\ox,\olm)$ be a solution to the KKT system \eqref{vs} and set $\ov:= \nabla_x L(\ox,\olm)$. Then the following conditions are equivalent: 
 \begin{enumerate}[noitemsep,topsep=0pt]
 \item  the SOSC \eqref{sosc} holds at $(\ox,\olm)$;
 \item   there exists a constant $\bar\rho>0$ such that for any $\rho\ge\bar\rho$ the second-order condition \eqref{eq26} is satisfied;
 \item there exist positive constants $\bar\rho$, $\gg$, and $\ell$ such that for any $\rho\ge\bar \rho$, the quadratic growth condition \eqref{eq16} 
holds.
\end{enumerate}
\end{Theorem}
\begin{proof} The equivalence between (b) and (c) was already established in Proposition~\ref{2ndgr}.  Suppose that (b) is satisfied. Since we always have 
$$
e_{1/{2\rho}}\big(\d^2 g(\Phi(\ox),\olm)\big)(\nabla \Phi(\ox)w)\le \d^2 g(\Phi(\ox),\olm)(\nabla \Phi(\ox)w)\quad \mbox{for all}\;\; \rho>0, 
$$
the validity of the second-order condition \eqref{eq402} yields the SOSC \eqref{sosc} and so proves (a). 

Assume now that (a) holds. Since the second subderivative is lower semicontinuous and positive homogeneous of degree $2$ (cf. \cite[Proposition~13.5]{rw}), the SOSC \eqref{sosc} amounts to the existence of 
a positive constant $\ell$ so that we have 
\begin{equation}\label{som}
\langle\nabla_{xx}^2L(\bar x,\olm)w,w\rangle+\d^2g (\Phi(\bar x),\olm) (\nabla \Phi(\ox)w )\ge \ell \quad \mbox{for all }\; w\in \D\cap \S,
\end{equation}
where $\S$ stands for the unit sphere in $\R^n$. Let $w\in \S$ and denote the left-hand side of \eqref{som} by $h(w)$.
  Given $\rho>0$ and $w\in \R^n$, define the function $h_\rho:\R^n\to \R$ by 
$$
h_\rho(w)=\big\langle \nabla^2_{xx}L(\ox,\olm)w,w\big\rangle+ e_{1/{2\rho}}\big(\d^2 g(\Phi(\ox),\olm)\big)(\nabla \Phi(\ox)w).
$$
We understand from \cite[Theorem~1.25]{rw} that 
$$
h_\rho(w) \nearrow h(w)\;\; as \;\; \rho \nearrow \infty
$$
 for all $w\in\R^n$. Employing now \cite[Proposition~7.4]{rw} tells us that  $h_\rho\overset {\ss \mbox{e}}{\to} h$ as $\rho \nearrow \infty$.
 This, combined with \cite[Theorem~7.46]{rw}, implies that $h_\rho+\dd_{\S\cap K_\Th(\ox,-\ov)}\overset {\ss \mbox{e}}{\to} h+ \dd_{\S\cap K_\Th(\ox,-\ov)}$ as  $\rho \nearrow \infty$.
 Appealing now to \cite[Theorem~7.33]{rw} indicates that 
\begin{equation}\label{infco}
 \inf \big\{  h_\rho(w)|\; w\in \S\cap K_\Th(\ox,-\ov)\big\} \to \inf\big\{ h(w)|\; w\in \S\cap K_\Th(\ox,-\ov)\big\}\;\;\mbox{as}\;\;\rho \nearrow \infty.
\end{equation}
It follows from \eqref{pwfor} and \eqref{domcr} that  for any $w\in \R^n$ with $\nabla \Phi(\ox)w \notin K_g(\Phi(\ox),\olm)$, we have  $\d^2g (\Phi(\bar x),\olm) (\nabla \Phi(\ox)w )=\infty$.
This,  \eqref{som}, and the definition of the convex cone $\D$ from \eqref{coned}  bring us to 
$$
\inf\big\{ h(w)|\; w\in \S\cap K_\Th(\ox,-\ov)\big\}\ge \ell.
$$
Combining this and \eqref{infco}, we can find a positive constant $ \bar \rho$ for which we have 
\begin{equation}\label{rb1}
h_\rho(w)>\frac{\ell}{2}\quad \mbox{for all}\;\;\rho\ge  \bar \rho\;\;\mbox{and all}\;\;w\in    \S\cap K_\Th(\ox,-\ov).
\end{equation}
Given $\rho\ge \bar\rho$, it follows from \eqref{eq402} and the definition of $h_\rho$  that 
$$
\d^2_x\L\big((\ox,\olm,\rho),\ov\big)(w) >0\quad \mbox{for all} \;\; w\in \S\cap K_\Th(\ox,-\ov).
$$
Remember that the second-order epi-derivative  is positive homogeneous of degree $2$. So we arrive at the second-order condition \eqref{eq26} for any $\rho\ge \bar\rho $. 
This proves (b) and hence completes the proof.
\end{proof}

The above result falls into a similar category of variational properties of   augmented Lagrangians, established recently in \cite{mmstams} 
for  constrained optimization problems. Our  proof of Theorem~\ref{growth}, however,  uses an idea
that was exploited  recently in \cite[Theorem~5]{r2020} and departs significantly from  the one in  \cite{mmstams} for constrained optimization  problems. 
Note that  the implication (a)$\implies$(b) in Theorem~\ref{growth} was derived 
  in \cite[Theorem~7.4]{r93}   for nonlinear programming problems without appealing to the second subderivative.  
 A similar characterization of the growth condition for augmented Lagrangians of parabolically regular constrained optimization 
 problems was obtained recently in \cite[Theorem~8.4]{mmstams}.
 
 Our final goal in this section is to achieve a uniform version of the growth condition \eqref{eq16}, which can be expressed equivalently using Proposition~\ref{fopag}(a) as
  \begin{equation}\label{eq16v2}
\L(x,\olm,\rho)\ge \L(\ox,\olm,\rho)+\ell\| x-\bar x\|^2\quad \mbox{for all}\;\;\rho\ge  \bar \rho\;\;\mbox{and all}\;\;\;x\in\B_\gg (\ox)\cap \Th
\end{equation}
Note that the constants $\bar\rho$, $\ell$, and $\gg$ depend on the Lagrange multiplier $\olm$. 
To establish a uniform version of \eqref{eq16v2} when $\olm$ therein is replaced by $\lm \in \Lm(\ox)$ that are sufficiently close to $\olm$, we 
require to show that  the constants $\bar\rho$, $\ell$, and $\gg$ in \eqref{eq16v2}   can be chosen to be independent of $\lm$ whenever $\lm$ is in $\Lm(\ox)$ and sufficiently close to   $\olm$. This goal will be achieved in 
two steps in the rest of this section.

\begin{Lemma}[second-order epi-derivatives of augmented Lagrangians] \label{usog}
Let $(\ox,\olm)$ be a solution to the KKT system \eqref{vs} and set $\ov:= \nabla_x L(\ox,\olm)$. 
Then the following conditions are equivalent: 
 \begin{enumerate}[noitemsep,topsep=0pt]
 \item the  SOSC \eqref{sosc} holds at $(\ox,\olm)$;
 \item there exist positive constants $\bar\rho$, $\ell_1$, $\ve_1$ such that for all $\rho\ge\bar\rho$ and
all  $\lm\in\Lambda(\ox)\cap\B_{\varepsilon_1}(\olm)$ we have
\begin{equation}\label{eq44}
\d^2_x\L \big((\ox,\lm,\rho), \nabla_x L(\ox,\lm)\big)(w)\ge\ell_1\|w\|^2\;\mbox{ for all }\;w\in K_\Th\big(\ox, - \nabla_x L(\ox,\olm)\big).
\end{equation}
\end{enumerate}
\end{Lemma}
\begin{proof} The implication (b)$\implies$(a)   results directly from Theorem~\ref{growth}. To prove the opposite implication, assume that (a) holds. 
It follows from Theorem~\ref{growth} that there exists a constant $\bar\rho>0$ such that for all $\rho\ge\bar\rho$ the second-order condition \eqref{eq26} is satisfied. 
Recall from  \cite[Proposition~13.5]{rw} that  the second-order epi-derivative  is lower semicontinuous and positive homogenous of degree $2$. 
Thus the second-order condition  \eqref{eq26} is equivalent  to the existence of a constant $\ell_1>0$ such that the estimate 
\begin{equation}\label{eq41}
\d^2_x\L \big((\ox,\olm,\bar\rho),\ov \big)(w)=\big\langle \nabla^2_{xx}L(\ox,\olm)w, w\big\rangle+ e_{ {1}/{2\bar\rho}}\big(\d^2 g(\Phi(\ox),\olm)\big)(\nabla\Phi(\ox)w)\ge 2\ell_1
\end{equation}
holds for all $w\in \S\cap K_\Th(\ox, -\ov)$, where the equality in \eqref{eq41} comes from  \eqref{eq40}. We proceed now by proving the following claim:

{\bf Claim A.} {\em There exists an $\varepsilon_1>0$ such that for any $\lm\in\Lambda(\ox)\cap\B_{\varepsilon_1}(\olm)$ the estimate
\begin{equation}\label{clq2}
\d^2_x\L \big((\ox,\lm,\bar\rho), \nabla_x L(\ox,\lm) \big)(w)=\big\langle \nabla^2_{xx}L(\ox,\lm)w,w\big\rangle+ e_{ {1}/{2\bar\rho}}\big(\d^2 g(\Phi(\ox),\lm)\big)(\nabla\Phi(\ox)w)\ge \ell_1
\end{equation} 
 holds for all $w\in \S\cap K_\Th\big(\ox, -\nabla_x L(\ox,\lm)\big)$.}
 
  To furnish this claim, choose $\varepsilon_1\in (0, {\ell_1}/{\|\nabla^2\Phi(\ox)\|})$ if $\|\nabla^2\Phi(\ox)\|\neq 0$ and $\varepsilon_1 > 0$ otherwise
  and let $w\in \S$.
 Thus for all     $\lm\in \B_{\ve_1}(\olm)$, we obtain   
\begin{eqnarray}\label{clq3}
\big\langle \nabla^2_{xx}L(\ox,\lm)w, w\big\rangle &=& \big\langle \nabla^2_{xx}L(\ox,\olm)w, w\big\rangle+\big\langle\lm -\olm, \nabla^2  \Phi (\ox)(w, w)\big\rangle\nonumber\\
&\geq&\big\langle \nabla^2_{xx}L(\ox,\olm)w, w\big\rangle - \|\nabla^2\Phi(\ox)\|\cdot\|\lm - \olm\|\nonumber\\
&\geq&\big\langle \nabla^2_{xx}L(\ox,\olm)w, w\big\rangle - \ell_1.
\end{eqnarray}
Moreover, it follows from \eqref{cc2} that 
\begin{eqnarray}\label{clq4}
e_{ {1}/{2\bar\rho}}\big(\d^2 g(\Phi(\ox),\lm)\big)(\nabla\Phi(\ox)w)&=&\inf_{v\in \R^m}\big\{\d^2g\big(\Phi(\ox)\big)(v)+\delta_{K_g(\Phi(\ox), \lm)}(v)+\bar\rho\|v-\nabla\Phi(\ox)w\|^2\big\}\nonumber\\ 
&=& \inf_{v\in K_g(\Phi(\ox), \lm)}\big\{\d^2g \big(\Phi(\ox) \big)(v)+\bar\rho\|v-\nabla\Phi(\ox)w\|^2\big\}. 
\end{eqnarray}
Note that $\d^2g(\Phi(\ox))$ does not depend on $\lambda$. Shrinking $\ve_1$ if necessary and recalling that $\sub g\big(\Phi(\ox)\big)$ is a polyhedral convex set, we conclude from   Proposition~\ref{sop}(a) that  
the relationships 
\begin{equation*}
K_g(\Phi(\ox), \lm) = N_{\partial g(\Phi(\ox))}(\lm) \subset N_{\partial g(\Phi(\ox))}(\olm) = K_g(\Phi(\ox), \olm) 
\end{equation*} 
hold for all $\lm\in\Lambda(\ox)\cap\B_{\varepsilon_1}(\olm)$.  These, together with  \eqref{clq4}, tell us  that for any $\lm\in\Lambda(\ox)\cap\B_{\varepsilon_1}(\olm)$ we have 
\begin{equation}\label{clq5}
e_{ {1}/{2\bar\rho}}\big(\d^2 g(\Phi(\ox),\lm)\big)(\nabla\Phi(\ox)w) \geq e_{ {1}/{2\bar\rho}}\big(\d^2 g(\Phi(\ox),\olm)\big)(\nabla\Phi(\ox)w).
\end{equation} 
 Combining \eqref{eq41}, \eqref{clq3}, and \eqref{clq5} verifies   \eqref{clq2} for all $w\in \S\cap K_\Th(\ox, -\ov)$ and   all $\lm\in\Lambda(\ox)\cap\B_{\varepsilon_1}(\olm)$. 
 To justify \eqref{clq2} for all $w\in \S\cap  K_\Th \big(\ox, -\nabla_x L(\ox, \lm) \big)$, observe that by the  polyhedrality of $\Th$,   we can assume without loss of generality that 
 $$
 K_\Th \big(\ox, -\nabla_x L(\ox, \lm) \big)=N_{N_\Th(\ox)}\big(-\nabla_x L(\ox, \lm)\big) \subset N_{N_\Th(\ox)}\big(-\nabla_x L(\ox, \olm)\big) = K_\Th(\ox, -\ov)
 $$
  for all $\lm\in\Lambda(\ox)\cap\B_{\varepsilon_1}(\olm)$. This clearly justifies Claim A. 
  
Using the positive homogeneity of the second-order epi-derivative yields \eqref{eq44} for $\rho = \bar\rho$ and for all $\lm\in\Lambda(\ox)\cap\B_{\varepsilon_1}(\olm)$. Recall that the function $$\rho\mapsto e_{ {1}/{2\rho}}\big(\d^2 g(\Phi(\ox),\lm)\big)(\nabla\Phi(\ox)w)$$ is nondecreasing. So is the function $\rho \mapsto\d^2\L( (\ox, \lm, \rho), \nabla_x L(\ox, \lm))(w)$. This  
yields  \eqref{eq44}  for all $\lm\in\Lambda(\ox)\cap\B_{\varepsilon_1}(\olm)$ and all $\rho \geq \bar\rho$, and hence completes the   proof.
\end{proof} 

We are now ready to achieve the uniform quadratic growth condition for the augmented Lagrangian \eqref{aug} under the SOSC \eqref{sosc}.

\begin{Theorem}[uniform quadratic growth condition for augmented Lagrangians]\label{ugrowth}
Let $(\ox,\olm)$ be a solution to the KKT system \eqref{vs}. Then the following conditions are equivalent:
 \begin{enumerate}[noitemsep,topsep=0pt]
 \item    the  SOSC \eqref{sosc} holds at $(\ox,\olm)$;
 \item there are positive constants $\bar\rho, \gamma, \varepsilon,\ell$ such that for all $\lm\in\Lambda(\ox)\cap\B_\varepsilon(\olm)$ and all $\rho\ge\bar\rho$  the uniform quadratic growth condition
\begin{equation}\label{eq14}
\L(x,\lm,\rho)\ge \ph(\ox)+g(\Phi(\ox))+\ell\| x-\ox\|^2\;\mbox{ for all }\;x\in\B_\gamma(\ox)\cap \Th
\end{equation}
is satisfied.
\end{enumerate}
\end{Theorem} 

\begin{proof} The implication (b)$\implies$(a)   results directly from Theorem~\ref{growth}. To prove the opposite implication, assume that (a) holds. 
It follows from Lemma~\ref{usog} that there exist the positive constants $\ell_1$, $\varepsilon_1$, and $\bar\rho$  for which \eqref{eq44} is satisfied  for all 
 $\lm\in\Lambda(\ox)\cap\B_{\varepsilon_1}(\olm)$ and all $\rho\ge\bar\rho$. Using this and  Proposition~\ref{fopag}(a), 
 we deduce from \eqref{eq44} that for any $\lm\in\Lambda(\ox)\cap\B_{\varepsilon_1}(\olm)$ there exists $\gamma_{\lm}>0$ for which we have 
\begin{equation}\label{eq22}
\L(x,\lm,\bar\rho)\ge \ph(\ox)+g(\Phi(\ox))+\dfrac{\ell_1}{2}\Vert x-\ox\Vert^2\;\textrm{ for all }\;x\in\B_{\gamma_\lm}(\ox)\cap \Th,
\end{equation}
where the constant $ {\ell_1}$ can be chosen the same for all the multipliers $\lm\in\Lambda(\ox)\cap\B_{\varepsilon_1}(\olm)$. The
latter results directly from the definition of the second subderivative. The radii of the balls centered at $\ox$ in \eqref{eq22}, however,  depend on $\lm$. We argue below that  a common radius can be chosen for all the multipliers $\lm\in\Lambda(\ox)$ that are sufficiently close to $\olm$. To this end, assume that 
$$
\tilde\B_{\varepsilon_1}(\olm)=\big\{\lm\in \R^m|\; \|\lm-\olm\|_1\le \ve_1\big\}
$$
where $\|\lm-\olm\|_1$ stands for $L_1$-norm of the vector $\lm-\olm$, namely the sum of the absolute values of the components of $\lm-\olm$. Clearly, we have $\tilde\B_{\varepsilon_1}(\olm)\subset \B_{\varepsilon_1}(\olm)$.  Define the function $\psi\colon\R^m\to\oR$ by
\begin{equation}\label{eq23}
\psi(\lm):=\sup_{x\in\big(\B_{\gamma_{\olm}}(\ox)\cap \Th\big)\setminus \{\ox\}}\frac{\ph(\ox)+g(\Phi(\ox))-\L(x,\lambda,\bar\rho)}{\Vert x-\ox\Vert^2}
+\delta_{\Lambda(\ox)\cap \tilde\B_{\varepsilon_1}(\olm)} (\lambda),\quad\lm\in\R^m.
\end{equation}
By Proposition~\ref{paug}(b), the function $\lambda\mapsto\L(x,\lm,\bar\rho)$ is concave. This, combined with the convexity of the set $\Lambda(\ox)\cap \tilde\B_{\varepsilon_1}(\olm)$, 
 guarantees  that $\psi$ in \eqref{eq23} is a convex function. Moreover, $\psi$ is lower semicontinuous since  the first term on the right-hand side of \eqref{eq23} is the pointwise supremum of a collection of continuous functions.
 Now we claim  that for any $\lm\in\Lambda(\ox)\cap \tilde\B_{\varepsilon_1}(\olm)$ the value $\psi(\lm)$ is finite. To justify this claim, let $\lm\in\Lambda(\ox)\cap \tilde\B_{\varepsilon_1}(\olm)$. If  $\gamma_\lm\ge\gamma_{\olm}$,  we conclude from    \eqref{eq22} that 
\begin{equation*}
\psi(\lm)\le\sup_{x\in \big(\B_{\gamma_{\lm}}(\ox)\cap \Th\big)\setminus \{\ox\}}\frac{\ph(\ox)+g(\Phi(\ox))-\L(x,\lm,\bar\rho)}{\Vert x-\ox\Vert^2}\le-\dfrac{\ell_1}{2},
\end{equation*}
which, in particular,   implies that $\psi(\olm)\le- {\ell_1}/{2}$. On the other hand, if $\gamma_\lm<\gamma_{\olm}$, we get  
\begin{eqnarray*}
 &&\psi(\lm) \le\\
&&  \max\left\{\sup_{x\in \big(\B_{\gamma_{\lm}}(\ox)\cap \Th\big)\setminus \{\ox\}}\frac{\ph(\ox)+g(\Phi(\ox))-\L(x,\lm,\bar\rho)}{\Vert x-\ox\Vert^2},\max_{\substack{\gamma_\lm\le\Vert x-\ox\Vert\le\gamma_{\olm}\\ x\in \Th}}\dfrac{\ph(\ox)+g(\Phi(\ox))-\L(x,\lm,\bar\rho)}{\Vert x-\ox\Vert^2}\right\}.
\end{eqnarray*}
By \eqref{eq22}, the first term inside the maximum is bounded above by  $- {\ell_1}/{2}$.  Also, the second term is finite since it is the maximum of a continuous function over a compact set. 
Combining these tells us that $\psi(\lm)$ is finite for all $\lm\in\Lambda(\ox)\cap \tilde\B_{\varepsilon_1}(\olm)$, which  results in 
$$
\dom\psi=\Lambda(\ox)\cap \tilde\B_{\varepsilon_1}(\olm).
$$
Since both sets  $\Lambda(\ox)$ and $\tilde\B_{\varepsilon_1}(\olm)$ are polyhedral, $\dom\psi$ is   a compact polyhedral convex set.
Appealing now to   \cite[Theorem~2.35]{rw} ensures that  $\psi$ is continuous relative to its domain.  
Thus, we can find  $\varepsilon >0$ such that $\B_\varepsilon(\olm)\subset \tilde\B_{\varepsilon_1}(\olm)$   and that
\begin{equation*}
\psi(\lm)\leq \psi(\olm)+\frac{\ell_1}{4}\leq -\frac{\ell_1}{4}  
\end{equation*} 
for all $\lm \in \dom\psi\cap \B_\varepsilon(\olm)=\Lambda(\ox)\cap \B_\varepsilon(\olm)$. This, combined with the definition of $\psi$ in \eqref{eq23}, tells us that for 
all $x\in\B_{\gamma_{\olm}}(\ox)\cap \Th$ and all $\lm\in\Lambda(\ox)\cap\B_\varepsilon(\olm)$ we have 
\begin{equation}\label{eq25}
\L(x,\lm,\bar\rho)\ge \ph(\ox)+g(\Phi(\ox))+\frac{\ell_1}{4}\Vert x-\bar x\Vert^2.
\end{equation}
Remember that by  Proposition~\ref{paug}(c) we always have 
\begin{equation*}
\L(x,\lm,\rho)\ge\L(x,\lm,\bar\rho)\;\mbox{ for all}\;\;\rho\ge\bar\rho.
\end{equation*}
Combining this with \eqref{eq25} and setting $\ell:= {\ell_1}/{4}$ verify the uniform quadratic growth condition \eqref{eq14} and hence complete the proof.
\end{proof} 

The uniform quadratic growth condition, established in Theorem~\ref{ugrowth}, plays an indispensable role in the local convergence analysis of the ALM
in the next section. A similar result for nonlinear programming problems was obtained in \cite[Proposition~3.1]{fs12} without appealing to  the concept 
of the second-order epi-derivative. 

We close this section by mentioning that it was recently observed by Rockafellar in \cite[Theorems~2 and 4]{r2020}  that 
when $g$ in \eqref{comp} is convex piecewise linear and $\Th=\R^n$, a stronger version of the uniform version of the quadratic growth condition \eqref{eq14}
for the augmented Lagrangian \eqref{aug} amounts to the {\em strong} version of the SOSC  \eqref{sosc}.  
The strong SOSC was introduced first by Robinson in \cite{rob} for nonlinear programming problems in order to justify the 
strong regularity of the KKT systems for this class of problems.
The uniform version of the quadratic growth condition in \cite{r2020} for the augmented Lagrangian is slightly different 
from the one, established in Theorem~\ref{ugrowth}. Indeed, the former does not require that $\lm$ in Theorem~\ref{ugrowth}
belong to the Lagrange multiplier set $\Lm(\ox)$. While this version of the quadratic growth condition is stronger than \eqref{eq14},
it demands strictly stronger version of the SOSC \eqref{sosc}. Moreover, we show in the next section that the local convergence 
analysis of the ALM for \eqref{comp} can be accomplished under a weaker version of the uniform quadratic growth condition, namely the condition \eqref{eq14},
and so one does not need to assume the strong version of the SOSC \eqref{sosc} for this task. 

\section{Local Convergence Analysis of ALMs}\label{sect05}\sce

This section concerns the local convergence analysis of the ALM for the composite optimization problem \eqref{comp}.
Our local convergence analysis relies heavily on the   error bound estimate \eqref{subr3} as well 
as the uniform quadratic growth condition \eqref{eq14}, which are satisfied for \eqref{comp} under  the  SOSC \eqref{sosc}.

Recall that for  the current triple $(\xk, \lk,\rok)\in \R^n\times \R^m\times (0,\infty)$, the next primal iterate $\xkk$ in the ALM is computed as an optimal solution to 
  the subproblem
\begin{equation}\label{conv1}
\mini \L(x, \lk, \rok)\quad \mbox{subject to}\quad x\in \Th.
\end{equation} 
The next dual iterate $\lkk$ is  then calculated as the Lagrange multiplier associated with the optimal solution $\xkk$ to the subproblem \eqref{conv1}, namely  
\begin{equation}\label{lkk}
 \lkk := \nabla\menvk\big(\Phi(\xkk)+ \rho_k^{-1}{\lk}\big).
\end{equation}
 Since an optimal solution  $\xkk$ to the subproblem \eqref{conv1} always satisfies the {\em stationary}  condition
\begin{equation}\label{conv2}
0\in \nabla_x\L(\xkk, \lk, \rok) +N_\Th(\xkk),
\end{equation} 
and since it is often difficult to find an exact solution to \eqref{conv1} in practice, it is more convenient to look for a solution $\xkk$ satisfying the {\em approximate} stationary condition
\begin{equation}\label{xkk}
\dist\big(-\nabla_x\L(\xkk, \lk, \rok),N_\Th(\xkk)\big)\leq \ek,
\end{equation} 
where $\ek\geq 0$ is a   tolerance parameter.
\begin{Algorithm}[augmented Lagrangian method]\label{Alg1}
Choose $(x^0,\lambda^0)\in\Th\times \R^m$ and $\bar\rho>0$. Pick a sequence of tolerance parameters $\{\ek\c$ with $\ek>0$ for all $k$ and $\ek\to   0$ as $k\to\infty$ and a sequence of $\{\rok\c$ such that  $\rok\ge \bar\rho$ for all $k$ and set $k:=0$. Then
\begin{itemize}[noitemsep,topsep=0pt]
\item [{\rm (1)}] if $(\xk,\lk)$ satisfies a suitable termination criterion, stop;
\item[{\rm (2)}]  otherwise, choose  $\xkk$ such that  \eqref{xkk} holds and then update the Lagrange multiplier  $\lkk$ by \eqref{lkk};
\item[{\rm (3)}] set $k\leftarrow k+1$ and go to Step~1.
\end{itemize}
\end{Algorithm} 

We begin our local convergence analysis of the ALM by showing that the solution mapping to the subproblem \eqref{conv1} 
is nonempty-valued and isolated calm uniformly in  $\rho$. Recall that a set-valued mapping $F:\R^n\tto\R^m$  is called isolated calm  at $(\ox,\oy)\in \gph F$ if there are
 a constant $\tau\in \R_+$ and a neighborhood $U$ of $\ox$ such that the inclusion  
 $$
 F(x)\cap V\subset \{\oy\}+\tau\|x-\ox\|\B\;\;\mbox{for all}\; x\in U
 $$
holds. 

\begin{Proposition}[solvability of subproblems]\label{solv} 
Let $(\ox,\olm)$ be a solution to the KKT system \eqref{vs} and let  the  SOSC \eqref{sosc} hold at $(\ox,\olm)$.
Then there exist positive constants      $\tau$ and $\hat\gamma\le \gg$  with $\gg$ taken from Theorem~{\rm \ref{ugrowth}}  such that for any $\rho\ge \bar\rho $ with $\bar\rho$ taken from Theorem~{\rm\ref{ugrowth}},
 the optimal solution mapping $S_\rho:\R^m\to \R^n$, defined by 
\begin{equation}\label{eq177}
S_\rho(\lambda):=\argmin\big\{\L(x,\lambda,\rho)\;\big|\; x\in\Th\cap\B_{\hat\gamma}(\ox)\big\}, \;\;\lm\in \R^m,
\end{equation}
enjoys the uniform  isolated calmness property 
\begin{equation}\label{pt9}
   S_\rho(\lambda) \subset\{\ox\} +\tau \Vert\lambda-\olm\Vert\B
\end{equation}
 and satisfies the condition $\emp \neq S_\rho(\lambda)\subset  \inte  \B_{\hat\gamma}(\ox)$ for all $\lm\in \B_{\hat\gamma/2\tau}(\olm)$.  
 \end{Proposition}

\begin{proof} Since the SOSC \eqref{sosc} is satisfied at $(\ox,\olm)$, it follows from Theorem~\ref{ugrowth}
that there exist positive constants $\bar\rho$, $\gg$, and $\ell$ such that for any $\rho\ge\bar \rho$, the uniform quadratic growth condition \eqref{eq14} holds for $\lm=\olm$. Moreover, 
by the twice differentiability of  $\Phi$  at $\ox$, we find some  constants $\hat\gamma\in(0,\gamma]$ and $\kappa_\Phi>0$ for which we have 
\begin{equation}\label{eq49}
\Vert\Phi(x)-\Phi(\ox)\Vert\le\kappa_\Phi\Vert x-\ox\Vert\;\textrm{ for all }\;x\in\B_{\hat\gamma}(\ox).
\end{equation}
Clearly,  the uniform quadratic growth condition \eqref{eq14} for $\lm=\olm$ indicates  that  $S_\rho(\olm)\cap \B_{\gamma}(\ox)=\{\ox\}$ for all $\rho\ge\bar\rho$. 
Define   the positive constant
\begin{equation}\label{eq21}
\tau:=\frac{\kappa_\Phi}{\ell}+\sqrt{\frac{\kappa_\Phi^2}{\ell^2}+\frac{1}{\ell\bar\rho}},
\end{equation}
and fix   $\lambda\in \B_{\hat\gamma/2\tau} (\olm)$ and $\rho\geq\bar\rho$.
Observe from the classical Weierstrass theorem that $S_\rho(\lm)\neq\emp$.   Assume that  $u\in S_\rho(\lambda)$.
By  Proposition~\ref{paug}(b),   the function $\lambda\mapsto\L(u,\lambda,\rho)$ is concave. 
By \cite[Theorem~2.26(b)]{rw}, we have 
\begin{equation}\label{ueq19}
\nabla_\lm\L(u, \lm, \rho)=\rho^{-1}\big(\nabla\menv\big(\Phi(u)+\rho^{-1}{\lm}\big) - \lm\big)
= \Phi(u)-\prox \big(\Phi(u)+ \rho^{-1}\lm\big),
\end{equation} 
 which in turn yields  the relationships
\begin{eqnarray*} 
\L(u,\lambda,\rho)&\ge&\L(u,\olm,\rho)-\langle\nabla_\lambda\L(u,\lambda,\rho),\olm-\lambda\rangle\nonumber\\
&=&\L(u,\olm,\rho)-\big\langle \Phi(u)-\prox\big(\Phi(u)+ \rho^{-1}{\lm}\big),\olm-\lambda\big\rangle\nonumber\\
&\ge& \ph(\ox)+g(\Phi(\ox))+\ell\Vert u-\ox\Vert^2-\big\langle \Phi(u)-\prox\big(\Phi(u)+ \rho^{-1}{\lm} \big),\olm-\lambda\big\rangle
\end{eqnarray*}
with the last inequality resulting from  \eqref{eq16}. We conclude from  $u\in S_\rho(\lambda)$ that 
\begin{equation*}
\L(u,\lambda,\rho)\le\L(\ox,\lambda,\rho)=\ph(\ox)+\menv\big(\Phi(\ox)+\rho^{-1}{\lm}\big) - \sm \rho^{-1}{\|\lm\|^2}\leq \ph(\ox)+g \big(\Phi(\ox) \big).
\end{equation*}
Combining these brings us to 
\begin{equation}\label{eq20}
\Vert u-\ox\Vert^2\le\frac{1}{\ell}\big\langle \Phi(u)-\prox\big(\Phi(u)+\rho^{-1}\lm\big),\olm-\lambda\big\rangle.
\end{equation}
By the first equality in \eqref{prm}, the inclusion  $\olm \in \partial g\big(\Phi(\ox)\big)$ amounts to $\Phi(\ox) = \prox\big(\Phi(\ox)+ \rho^{-1}\olm\big)$. This leads us to 
\begin{eqnarray}\label{clq7}
&&\big\Vert\Phi(u)-\prox\big(\Phi(u)+\rho^{-1}\lm\big)\big\Vert \nonumber\\
&=&\big\Vert\Phi(u)-\Phi(\ox)+\prox\big(\Phi(\ox)+\rho^{-1}\olm\big)-\prox\big(\Phi(u)+\rho^{-1}\lm\big)\big\Vert\nonumber\\
&\le&2\Vert\Phi(u)-\Phi(\ox)\Vert+\rho^{-1}{\Vert\olm-\lambda\Vert}\nonumber\\
&\le&2\kappa_\Phi\Vert u-\ox\Vert+\rho^{-1}{\Vert\olm-\lambda\Vert},
\end{eqnarray}
where the first inequality comes from  the fact that proximal mapping of a convex function is always nonexpansive (cf.  \cite[Proposition~12.19]{rw}), 
and where the last inequality results from \eqref{eq49}. This, combined with  \eqref{eq20},  brings us   to 
\begin{equation*}
\Vert u-\ox\Vert^2\le\frac{1}{\ell}\Big(2\kappa_\Phi\Vert u-\ox\Vert+\rho^{-1}{\Vert\lambda-\olm\Vert}\Big)\Vert\lambda-\olm\Vert.
\end{equation*}
Since the latter inequality  can be expressed  equivalently as 
\begin{equation*}
\ell\Vert u-\ox\Vert^2-2\kappa_\Phi\Vert \lambda-\olm\Vert\cdot\Vert u-\ox\Vert-\rho^{-1}{\Vert \lambda-\olm\Vert^2 }\le 0,
\end{equation*}
we arrive at  the estimate
\begin{equation*}
\Vert u-\ox\Vert\le\left(\dfrac{\kappa_\Phi}{\ell}+\sqrt{\frac{\kappa_\Phi^2}{\ell^2}+\frac{1}{\ell\rho}}\right)\Vert\lambda-\olm\Vert\le\tau\Vert\lambda-\olm\Vert \le \tau \hat\gamma/2\tau < \hat\gamma.
\end{equation*}
This clearly justifies   \eqref{pt9}   and  the inclusion $S_\rho(\lambda)\subset  \inte  \B_{\hat\gamma}(\ox)$ for all  $\lambda\in \B_{\hat\gamma/2\tau} (\olm)$ and $\rho\geq\bar\rho$,  and hence completes the proof.
\end{proof} 

\begin{Remark}{\rm Assume that  the SOSC \eqref{sosc} holds at $(\ox,\olm)$. It is not hard to see that a similar argument as those for \eqref{eq20}
shows that for any $\lm\in \R^m$, any $u\in S_\rho(\lm)$,  any $\mu \in \Lm(\ox)\cap \B_\ve(\olm)$, and any $\rho\ge \bar\rho$ we have 
\begin{equation*} 
\Vert u-\ox\Vert^2\le\frac{1}{\ell}\big\langle \Phi(u)-\prox\big(\Phi(u)+\rho^{-1}\lm\big),\mu-\lambda\big\rangle,
\end{equation*}
where constants $\bar\rho$, $\ve$, and $\ell$ come from Theorem~\ref{ugrowth}.
Combining this and the second equality in \eqref{ueq19} brings us to 
\begin{equation}\label{ueq20}
\Vert u-\ox\Vert^2\le\frac{1}{\ell\rho}\big\langle \nabla e_{1/\rho}\big(\Phi(u)+\rho^{-1}\lm\big)-\lm,\mu-\lambda\big\rangle,
\end{equation}
which will be utilized in the proof of Theorem~\ref{est} later in this section.
}
\end{Remark}

Next, we are going to establish an error bound estimate for the consecutive iterates of the ALM. 
In order to achieve this goal, the polyhedral set $\Th$ in \eqref{comp} has to be an affine set. We will explain 
after our proof of the aforementioned error bound estimate   why such a restriction on $\Th$ is required in our proof. 
Thus, in the rest of this paper, we assume further that the  polyhedral convex set  $\Th$ has a representation of the form 
\begin{equation}\label{Th}
\Th:=\big\{x\in\R^n\,|\, Bx = b\big \},
\end{equation} 
where $B$ is a $s\times n$ matrix and $b$ is a vector in $\R^s$. Furthermore, our convergence analysis of the ALM  requires to consider the \textit{residual function} $R: \R^n\times \R^m \to \R$ of the KKT system \eqref{vs}, given by
\begin{equation}\label{res}
R(x, \lm) : =  \dist\big(-\nabla_x L(x, \lm), N_\Th(x)\big)+\big\|\Phi(x)-{\rm{prox}}_g\big(\Phi(x)+\lm\big)\big\|, \quad (x, \lm) \in \R^n\times \R^m.
\end{equation}
We should add here that the residual function $R$ is taken from the error bound estimate \eqref{subr3}, which by Theorem~\ref{sooc} is satisfied 
under the SOSC \eqref{sosc}. 
It is not hard to see that  $R(\ox, \olm)=0$ whenever $(\ox, \olm)$ is a solution to the KKT system \eqref{vs}.  
While the estimate \eqref{subr3} requires the SOSC \eqref{sosc}, the next observation shows that the opposite inequality in \eqref{subr3}
holds without such a requirement. 

\begin{Proposition}[estimate for residual functions] \label{erct}  Assume that $(\ox,\olm)$ is a solution to the KKT system \eqref{vs} with $\Th$ taken from \eqref{Th}. Then 
there exist positive constants $\gamma_2$ and $\kappa_2$ such that
\begin{equation}\label{eb1}
R(x, \lambda) \leq \kappa_2\big(\|x-\ox\|+\dist(\lambda, \Lambda(\ox))\big)\quad\textrm{for all }\;\; (x, \lambda)\in \B_{\gamma_2}(\ox, \olm).
\end{equation} 
\end{Proposition}

\begin{proof} By \eqref{Th},  we have $N_\Th(x) = N_\Th(\ox)$ for all $x\in \Th$. Let $\mu=P_{\Lm(\ox)}(\lm)$
and observe that  $(\ox,\mu)$ is a solution to the KKT system \eqref{vs}. So, we know from \eqref{prm} that    $\mu \in \partial g \big(\Phi(\ox)\big)$ yields   $\Phi(\ox) = {\rm prox}_g\big(\Phi(\ox)+ \mu\big)$.
Using these tells us that 
\begin{eqnarray*}
R(x, \lambda)&=& \dist\big(-\nabla_x L(x, \lm), N_\Th(x)\big)+\big\|\Phi(x)-{\rm{prox}}_g\big(\Phi(x)+\lm\big)\big\|\\
&\le & \|\nabla_x L(x, \lm)-\nabla_x L(\ox, \olm)  \|+  \dist\big(-\nabla_x L(\ox, \olm), N_\Th(\ox)\big)\\
&&+  \|\Phi(x)-\Phi(\ox)\|+ \big\|{\rm{prox}}_g\big(\Phi(x)+\lm\big)-{\rm{prox}}_g\big(\Phi(\ox)+\mu\big)\big\|\\
&\le &\|\nabla_x L(x, \lm)-\nabla_x L(\ox, \olm)  \| +  2\|\Phi(x)-\Phi(\ox)\|+ \dist(\lambda, \Lambda(\ox)).
\end{eqnarray*}
This, together with \eqref{eq49},  ensures the existence of positive 
constants  $\gamma_2$ and $\kappa_2$ for which the estimate \eqref{eb1} is satisfied. 
\end{proof}

We are now in a position to establish an error bound estimate for the consecutive iterates of the ALM for the composite problem 
\eqref{comp} under the SOSC \eqref{sosc}.

\begin{Theorem}[error bound for consecutive iterates of ALM]\label{est} 
Let $(\ox,\olm)$ be a solution to the KKT system \eqref{vs} and  let the SOSC \eqref{sosc} hold at $(\ox,\olm)$. Then there exist positive constants   $\gamma_3$, and $\kappa_3$ such that for any $\rho\geq \bar\rho$ with 
$\bar\rho$ taken from Theorem~{\rm\ref{ugrowth}}, any $(x,\lambda)\in\B_{\gamma_3}(\ox,\olm)$ with $x\in \Th$ and  $R(x,\lambda)>0$, and any   optimal solution $u$ to the regularized problem 
\begin{equation}\label{regp}
\mini \L(w, \lm, \rho)\quad \mbox{subject to}\quad w\in \Th\cap \B_{\hat\gg}(\ox)
\end{equation} 
with $\hat\gg$ taken from Proposition~{\rm\ref{solv}},   the error bound estimate
\begin{equation}\label{est0}
\Vert u-x\Vert+\big\Vert\nabla\menv\big(\Phi(u)+ \rho^{-1}{\lm}\big)-\lm\big\Vert\le\kappa_3\,R(x,\lambda)
\end{equation}
holds, where the residual function  $R$ is  taken from \eqref{res}.
\end{Theorem}

\begin{proof} It follows from the SOSC \eqref{sosc} and    Theorem~\ref{ugrowth} that there are positive  constants $\bar\rho$, $\gamma$, $\varepsilon$, and $\ell$ for which 
the uniform quadratic  growth condition \eqref{eq14} holds  for all $\rho\ge\bar\rho$ and all $\lambda\in\Lambda(\ox)\cap\B_\varepsilon(\olm)$. 
According to Proposition~\ref{solv}, the solution mapping to \eqref{regp}, denoted by $S_\rho(\lm)$ therein, satisfies the uniform isolated calmness \eqref{pt9} and the condition $S_\rho(\lm)\subset  \inte  \B_{\hat\gamma}(\ox)$ for all $\lm\in \B_{\hat\gamma/2\tau}(\olm)$ and all $\rho\ge \bar\rho$, where both positive constants $\tau$ and $\hat\gg$ are taken from this proposition. This, in particular, tells us that for every $\lm\in \B_{\hat\gamma/2\tau}(\olm)$ and every $\rho\ge \bar\rho$
any optimal solution $u$ to \eqref{regp} satisfies the first-order optimality condition 
\begin{equation}\label{optc}
0\in \nabla_x\L(u, \lm, \rho)+ N_\Th(u).
\end{equation}
After this presentation, assume by contradiction that the error bound estimate \eqref{est0} fails. The latter implies that  there exists a sequence $\{(\xk, \lk, \rok)\b\subset \Th \times \R^m\times [\bar\rho, \infty)$ 
with $\xk\to\ox, \lk\to\olm$ as $k\to\infty$ such that
\begin{equation}\label{est1}
\Vert u^k-\xk\Vert+\| d^k-\lk\|>kR_k\quad \mbox{with}\;\;d^k:=\nabla\menvk\big(\Phi(u^k)+\rho_k^{-1}\lk\big),
\end{equation} 
where $u^k$ is an optimal solution to \eqref{regp} for $(\lambda, \rho):=(\lk, \rok)$, and where $R_k:=R(\xk, \lk)$ for each $k \in \mathbb{N}$. Denote by $\beta_k$ the left-hand side of \eqref{est1}, and thus get $R_k=o(\beta_k)$.
 By the definition of $R_k$, the latter yields
\begin{equation}\label{est2}
-\nabla_xL(\xk, \lk)+o(\xik)\in N_\Th(\xk)\quad\textrm{and}\quad \Phi(\xk)+o(\xik)= {\rm{prox}}_g\big(\Phi (\xk)+\lk\big).
\end{equation}
Passing to a subsequence if necessary, we can find $(\xi,\eta)\in \R^n\times \R^m$ such that 
\begin{equation}\label{est4}
\dfrac{u^k-\xk}{\xik}\to \xi\quad \mbox{and}\quad  \dfrac{ d^k-\lk}{\xik} \to  \eta \quad \mbox{with}\;\; (\xi,\eta)\neq 0.
\end{equation}
We know from Theorem~\ref{sooc} that the error bound \eqref{subr3} holds for all $(x,\lm)$ sufficiently close to  $(\ox, \olm)$. 
Set $\mu^k:=P_{\Lambda(\ox)}(\lk)$ and remember  that $(\xk, \lk)\to(\ox, \olm)$.  By  \eqref{subr3}, we can assume without loss of generality that $\xk-\ox=O(R_k)$ and $\lk-\mu^k=O(R_k)$ for all $k\in \N$, which in turn results in
\begin{equation}\label{est3}
\xk-\ox=o(\xik)\quad\textrm{and}\quad\lk-\mu^k=o(\xik)\quad\textrm{as } k\to\infty.
\end{equation}
The latter along with $\lk\to \olm$ tells us that $\mu^k\to \olm$. So, we get   $\mu^k\in\Lambda(\ox)\cap\B_\varepsilon(\olm)$ for all $k $ sufficiently large. 
This, combined with  $\rho_k\ge\bar\rho$ and \eqref{ueq20}, implies that 
\begin{eqnarray}\label{est5}
\Vert u^k-\ox\Vert^2&\le&\dfrac{1}{\rok\ell} \langle d^k-\lk,\mu^k-\lk \rangle\le\dfrac{1}{\rok\ell}\| d^k-\lk\|\cdot\Vert\mu^k-\lk\Vert,
\end{eqnarray}
which together with \eqref{est4} and  \eqref{est3} results in 
\begin{equation*}
\dfrac{\Vert u^k-\ox\Vert^2}{\xik^2}\le\dfrac{1}{\rok\ell}\dfrac{\| d^k-\lk\|}{\xik}\cdot\dfrac{\Vert\mu^k-\lk\Vert}{\xik}\to 0\quad\textrm{ as }\;k\to\infty.
\end{equation*}
Thus, we get $u^k-\ox=o(\beta_k)$. Combining this with \eqref{est3} clearly shows that 
\begin{equation*}
\xi=\lim_{k\to\infty}\dfrac{u^k-\xk}{\xik}=\lim_{k\to\infty}\dfrac{u^k-\ox}{\xik}-\lim_{k\to\infty}\dfrac{\xk-\ox}{\xik}=0-0=0.
\end{equation*}
If either the sequence $\{\rok\b$ or  $\{{\rok}/{\xik}\b$ is bounded, it follows from $u^k-\ox=o(\beta_k)$ that  $\dfrac{\rok}{\xik}\Vert u^k-\ox\Vert\to0$ as $k\to\infty$. 
This, \eqref{ueq19}, and a similar argument as those for \eqref{clq7} bring us to 
\begin{eqnarray*}
 \dfrac{ \|d^k-\lk\|}{\xik}
&=&\frac{\rok}{\xik} \big\Vert \Phi(u^k)-\proxk\big(\Phi(u^k)+\rho_k^{-1}\lk\big)\big\Vert\\
&\leq& 2\kappa_\Phi\frac{\rok}{\xik}\|u^k-\ox\|+\dfrac{\Vert \mu^k-\lk\Vert }{\xik}\to0\quad\textrm{as } k\to\infty,  
\end{eqnarray*} 
which by \eqref{est4} yields $\eta =0$. This is  a contradiction with \eqref{est4} since we showed that $(\xi,\eta)=0$.

Assume now that both sequences $\{\rok\b$ and  $\{{\rok}/{\xik}\b$  are unbounded.  We can assume by passing  to a subsequence if necessary that 
\begin{equation}\label{unb}
\rok\to \infty\quad \mbox{and}\quad \frac{\rok}{\beta_k}\to \infty\;\;\mbox{as}\;\; k\to \infty.
\end{equation}
Since $u^k$ is an optimal solution to \eqref{regp} associated with $(\lk,\rok)$, we deduce from \eqref{optc} that 
\begin{equation*}
0\in \nabla_x\L(u^k, \lk, \rok)+ N_\Th(u^k).
\end{equation*} 
By \eqref{Th}, we get  $N_{\Th}(u^k)=N_\Th(\xk) = \rge B^*$. These along with  \eqref{est2} tell us that 
\begin{equation*}
\nabla_x\L(u^k, \lk, \rok)-\nabla_xL(\xk, \lk)+o(\xik)\in N_\Th(\xk)-N_\Th(u^k)= \rge B^*.
\end{equation*}
Using the definitions of the augmented Lagrangian $\L$ and  $L$, we get 
\begin{eqnarray*}
\rge B^*&\ni& \nabla_x\L(u^k, \lk, \rok)-\nabla_xL(\xk, \lk)+o(\xik)\\
&=&\nabla \varphi(u^k)-\nabla \varphi(\xk)+\nabla\Phi(u^k)^* \nabla\menvk\big(\Phi(u^k)+\rho_k^{-1}\lk\big)-\nabla\Phi(\xk)^*\lk+o(\xik)\\
&=&\big(\nabla\Phi(u^k)-\nabla\Phi(\xk)\big) ^*\lk+\nabla\Phi(u^k)^*\big (d^k-\lk\big )+o(\xik)\\
&=&\nabla\Phi(u^k)^* \big(d^k-\lk\big)+o(\xik),
\end{eqnarray*}
where the second and last equalities result from  the Lipschitz continuity of  $\nabla \varphi$ and $\nabla \Phi$ around $\ox$ and from  the fact that $u^k-\xk=o(\xik)$. 
Dividing both sides  by $\xik$ and then letting  $k\to\infty$ confirm via \eqref{est4} that 
\begin{equation*}
\nabla\Phi(\ox)^*\eta \in\rge B^*.
\end{equation*} 
We are going to show that $\eta=0$, which contradicts \eqref{est4} and thus completes the proof. To furnish this, set $E:=\big\{z \in \R^n\, |\, \nabla \Phi(\ox)^*z\in \rge B^*\big\}$
and observe from \cite[Corollary~11.25(d)]{rw} that $E^\perp=\big\{\nabla\Phi(\ox)y|\;y\in \ker B\big\}$. So, we conclude  from  $\ox,u^k\in \Th$ and  \eqref{Th} that $\nabla\Phi(\ox)(u^k-\ox)\in E^\perp$.
Moreover, we obtain  from \eqref{est4}-\eqref{est5} that 
\begin{equation}\label{est6}
\dfrac{\rok}{\xik}\Vert u^k-\ox\Vert^2\leq\dfrac{1}{\ell} \dfrac{\| d^k-\lk\|}{\xik} \cdot\Vert \mu^k-\lk\Vert  \to 0 \quad\textrm{ as }k\to\infty.
\end{equation} 
Since $P_E$ is a linear mapping due to $E$ being a subspace, we arrive at 
\begin{eqnarray}\label{est7}
P_E\Big(\dfrac{\rok}{\xik}\big (\Phi(u^k)-\Phi(\ox)\big )\Big)&=&P_E\Big (\dfrac{\rok}{\xik}\nabla\Phi(\ox) (u^k-\ox )+O\big (\dfrac{\rok}{\xik}\Vert u^k-\ox\Vert^2 \big)\Big  )\nonumber\\
&=&P_E \Big(O\big(\dfrac{\rok}{\xik}\Vert u^k-\ox\Vert^2 \big )\Big)\to 0 \quad \textrm{as } k\to\infty,
\end{eqnarray} 
where the last equality results from $\nabla\Phi(\ox)(u^k-\ox)\in E^\perp$.
Fix $k\in \N$ and set  
$
z^k:= \Phi(u^k)+ \rho_k^{-1}\big(\lk -d^k\big).
$
It follows from $d^k=\nabla\menvk\big(\Phi(u^k)+ \rho_k^{-1}{\lk}\big)$ and the second equation in \eqref{prm} that 
\begin{equation}\label{clq8}
d^k \in \partial g(z^k).
\end{equation} 
Using  \eqref{est4} and \eqref{unb} allows us to arrive at 
\begin{equation*}
z^k =\Phi(u^k) - \frac{d^k-\lk}{\xik}\cdot\frac{\xik}{\rok} \to \Phi(\ox)=:\bar z\quad\textrm{as }\; k\to\infty.
\end{equation*} 
By \eqref{clq8}, we get $z^k\in \dom g=\cup_{i=1}^s C_i$ for all $k$, where the polyhedral convex sets $C_i$ are taken from \eqref{PWLQ}. 
By passing to a subsequence if necessary, we can assume without loss of generality that  for some $\oi\in I(\oz)$, we have $z^k \in C_\oi$ for all $k$, where the index set  $  I(\bar z)$  is defined by \eqref{act}.
Since $C_\oi$ is  polyhedral, we find   $l_\oi\in \N$, $c_{\oi, j}\in \R^m$, and $\al_{\oi, j}\in \R$,  $j=1,\ldots,l_\oi$,  such that 
\begin{equation*}
C_\oi= \big\{z\in \R^m\, |\, \langle c_{\oi, j}, z\rangle \leq \al_{\oi, j}, \quad j= 1, \ldots, l_\oi\big\}.
\end{equation*} 
Denote by $J(\bar z)$ and $J(z^k)$ the sets of the indexes of active inequalities in $C_\oi$ at $\bar z$ and $z^k$, respectively. Since $z^k\to \oz$, we obtain  $J(z^k)\subset J(\bar z)$ for all $k$ sufficiently large. 
Passing to a subsequence, if necessary, we can assume that there exists a subset $\bar J\subset J(\bar z)$ such that $J(z^k)=\bar J$ for all $k$. Thus, we have 
\begin{equation*}
z^k -\bar z \in \big\{v\in \R^m\, \big|\, \langle c_{\oi, j}, v\rangle= 0\;\textrm{ for } \;j \in \bar J, \; \langle c_{\oi, j}, v\rangle\leq 0\; \textrm{ for }\; j\in J(\bar z)\setminus \bar J \big\}=: \Omega.
\end{equation*} 
It follows from  the definitions of $z^k$ and $\bar z$ that 
\begin{equation*}
\Phi(u^k)-\Phi(\ox)-\frac{ d^k-\lk}{\rok}=z^k-\oz \in \Omega. 
\end{equation*}
Multiplying by $ {\rok}/{\xik}$ and using the linearity of  $P_E$  allow us to get 
\begin{equation*}
P_E\Big(\dfrac{\rok}{\xik}\big (\Phi(u^k)-\Phi(\ox)\big )\Big)- P_E \big(\frac{ d^k-\lk}{\xik}\big) \in P_E(\Omega),
\end{equation*} 
which, together with \eqref{est4}, \eqref{est7}, and the closedness of $P_E(\Omega)$, yields $-P_E(\eta) \in P_E(\Omega)$. Since $\eta\in E$, the latter confirms that 
\begin{equation}\label{est8}
-\eta=-P_E(\eta) = P_E(\zeta)\quad\textrm{for some }\;\; \zeta\in \Omega.
\end{equation}
Remembering $\mu^k\in \Lm(\ox)$, we obtain $\mu^k \in \partial g(\bar z)$. By the representation of $\partial g(\bar z)$ from \eqref{sub} and the fact that  $\oi\in I(\bar z)$, we conclude that 
\begin{equation*}\label{clq9}
\mu^k - A_\oi\bar z -a_\oi \in N_{C_\oi}(\bar z) = \cone\big\{c_{\oi, j}\, | \;j\in J(\bar z)\big\}.
\end{equation*} 
Since we also have $\oi\in I(z^k)$, it follows again from \eqref{clq8} and  \eqref{sub} that 
\begin{equation*}\label{clq10}
d^k -A_\oi z^k-a_\oi \in N_{C_\oi}(z^k) = \cone\big\{c_{\oi, j}\,|\; j\in \bar J\big\}.
\end{equation*} 
Combining these results in 
\begin{eqnarray*}
 \mu^k-d^k + A_\oi(z^k-\bar z)&\in&   \cone\big\{c_{\oi, j}\, | \;j\in J(\bar z)\big\}-\cone\big\{c_{\oi, j}\,|\; j\in \bar J\big\}\\
&=&\span\big\{c_{\oi, j}\,|\;  j\in \bar J\big\}+\cone\big\{c_{\oi, j}\,|\; j\in J(\bar z)\setminus\bar J\big\} =  \Omega^*.
\end{eqnarray*} 
Since $\lk -\mu^k=o(\xik)$ and $u^k-\ox=o(\beta_k)$, we deduce from \eqref{est4} and \eqref{unb} that 
\begin{eqnarray*}
\frac{ \mu^k-d^k}{\beta_k}+A_\oi\big(\frac{z^k-\oz}{\beta_k}\big)&=& \frac{ \mu^k-\lk}{\beta_k}+\frac{ \lk-d^k}{\beta_k} +A_\oi\big(\frac{\Phi(u^k)-\Phi(\ox)}{\beta_k}\big)\\
&&-\frac{1}{\rok}A_\oi\big(\frac{d^k-\lk}{\beta_k}\big)\to -\eta \;\;\;\;\mbox{as}\;\;k\to \infty,
\end{eqnarray*}
which in turn demonstrates that 
$-\eta \in \Omega^*$. This, combined with \eqref{est8} and $\eta\in E$, shows that 
\begin{equation*}
\|\eta\|^2= \langle \eta, \eta \rangle = \langle \eta, P_E(\eta) \rangle= -\langle \eta, P_E(\zeta) \rangle= -\big\langle \eta, P_E(\zeta)+P_{E^{\perp}}(\zeta) \big\rangle = -    \langle \eta, \zeta\rangle \le 0,
\end{equation*} 
implying that  $\eta =0$. This clearly contradicts \eqref{est4} and thus completes the proof.
\end{proof} 

As promised before, we want to discuss the reasons that  forced us to assume the representation \eqref{Th} for the polyhedral convex set $\Th$ in \eqref{comp}. 
The first place that requires such a representation of $\Th$ is Proposition~\ref{erct}. The other one happens in  the proof of Theorem~\ref{est}, where we
 require that the projection mapping of $P_E$ for the set $E$ in this proof be linear. If we do not assume \eqref{Th} for $\Th$, 
we can only ensure that the projection mapping $P_E$ is piecewise linear, which is not enough in our proof. Note also  that   most of the proof 
of Theorem~\ref{est} does not require that the function  $g$ be CPLQ. In fact, only the last part of the latter proof, which deals with the set $\O$, utilizes the peculiar geometry of CPLQ functions. 
Observe that the set $\O$ in this proof is   a critical   cone of the polyhedral convex set $C_\oi$ at $\oz$ for some $\hat\lm\in \Lm(\ox)$, which  can be different from $\olm$,  and so  is a face of $T_{C_\oi}(\oz)$. Since $C_\oi$ is polyhedral, there are only finitely many faces of $T_{C_\oi}(\oz)$, a fact that plays a key role in our proof and makes it difficult to extend this result for nonpolyhedral optimization problems. A similar result was established for NLPs in \cite[Proposition~3.3]{fs12}, namely when $g=\dd_{\R^m_-}$ and $\Th=\R^n$ in the composite problem \eqref{comp}.
However, the Lagrange multiplier in \cite[Proposition~3.3]{fs12} is restricted to belong to $\R^m_+$. As shown in our proof,  this  can be dropped with no harm by a small modification in the proof. 

\begin{Remark}{\rm Given $(\xk,\lk)\in  \B_{ \gg_3}(\ox,\olm)$ and $\rok\ge \bar \rho$ with $\gg_3$ and $\bar\rho$ taken from Theorem~\ref{est},   observe from Theorem~\ref{est} that 
any optimal solution $x^{k+1}$ to the subproblem \eqref{regp} for $\lm=\lk$ and $\rho=\rok$ satisfies  the error bound estimate 
 \begin{equation}\label{cone2}
 \|x^{k+1}-x^k\|+\|\lm^{k+1}-\lk\|\le \kappa_3 R(x^k,\lk)\quad \mbox{with}\;\;\lm^{k+1}:=\nabla\menv\big(\Phi(x^{k+1})+ \rok^{-1}{\lm}\big),
 \end{equation}
 where $\kappa_3$ comes from \eqref{est0}. Increasing $\kappa_3$ if necessary, one can see that both  \eqref{xkk} and \eqref{cone2} are satisfied for 
 any $\tilde x^{k+1}\in \Theta$ sufficiently close to $x^{k+1}$. In fact, we can find a constant $\ell \ge 0$ such that any $\tilde x^{k+1}\in \Theta$ sufficiently close to $x^{k+1}$, we can conclude 
that  $N_\Th(\tilde x^{k+1})=N_\Th(  x^{k+1})$ due to \eqref{Th} and that 
 \begin{eqnarray*}
&& \dist\big(-\nabla_x\L(\tilde x^{k+1}, \lk, \rok),N_\Th(\tilde x^{k+1})\big)=    \dist\big(-\nabla_x\L(\tilde x^{k+1}, \lk, \rok),N_\Th(x^{k+1})\big)\\
 &\le& \dist\big(-\nabla_x\L(\xkk, \lk, \rok),N_\Th(\xkk)\big) +\|\nabla_x\L(\tilde x^{k+1}, \lk, \rok)-\nabla_x\L(  x^{k+1}, \lk, \rok)\|\\
 &=& \|\nabla_x\L(\tilde x^{k+1}, \lk, \rok)-\nabla_x\L(  x^{k+1}, \lk, \rok)\| \le \|\nabla \ph(\tilde x^{k+1})- \nabla \ph(  x^{k+1})\|\\
 &&+ \|\nabla \Phi(\tilde x^{k+1})^*\nabla\menvk\big(\Phi(\tilde  x^{k+1})+ \rok^{-1}{\lm}^{k}\big)- \nabla \Phi(  x^{k+1})^*\nabla\menvk\big(\Phi(x^{k+1})+ \rok^{-1}{\lm}^{k}\big)\|\\
  &\le & \ell \|\tilde x^{k+1} -  x^{k+1}\|+ \ell\rok\| \tilde x^{k+1} -  x^{k+1}\| + \ell \|\lm^{k+1}\| \| \tilde x^{k+1} -  x^{k+1}\|,
\end{eqnarray*}
where the second equality comes from $x^{k+1}$ being an optimal solution to the subproblem \eqref{regp} for $\lm=\lk$ and $\rho=\rok$. 
This clearly proves \eqref{xkk} for a given tolerance parameter $\ve_k>0$ provided that $\tilde x^{k+1}$ is chosen so that 
$$
\big(\ell  + \ell\rok + \ell \|\lm^{k+1}\|\big)\| \tilde x^{k+1} -  x^{k+1}\|  \le \ve_k.
$$
Similarly, the error bound  \eqref{cone2} can be achieved  for  $\tilde x^{k+1}$ with $\kappa_3$ replaced with $2\kappa_3$ if    $\tilde x^{k+1}$ is chosen so that 
$$
(1+\ell'\rok)\| \tilde x^{k+1} -  x^{k+1}\| \le  \kappa_3 R(x^k,\lk),
$$
where $\ell'\ge 0$ is a Lipschitz constant for $\Phi$ around $\ox$. This, in fact, results from 
\begin{eqnarray*}
&& \| \tilde x^{k+1}-x^k\|+\|\nabla\menvk\big(\Phi(\tilde  x^{k+1})+ \rok^{-1}{\lm}^{k}\big) -\lk\| \\
&\le& \| \tilde x^{k+1} -  x^{k+1}\|  +\| \nabla\menvk\big(\Phi(\tilde  x^{k+1})+ \rok^{-1}{\lm}^{k}\big) -\lm^{k+1}\| + \|   x^{k+1} -  x^{k}\|+\| \lm^{k+1}-\lm^{k}\|\\
&\le & (1+\ell'\rok)\| \tilde x^{k+1} -  x^{k+1}\| + \kappa_3 R(x^k,\lk) \le 2\kappa_3 R(x^k,\lk).
\end{eqnarray*}
}
\end{Remark}

We are now ready to prove the main result of this section, namely the linear convergence of the inexact ALM.
 The proof exploits an iterative framework, proposed by Fisher in \cite[Theorem~1]{Fis02}, to achieve the superlinear
 convergence of generalized equations with non-isolated solutions. To this end, let  $\sigma:\R^n\times \R^m\to \R_+$ be a function that satisfies  the condition
 \begin{equation}\label{error}
 \sigma(x,\lm)=o\big(R(x,\lm)\big)\quad \mbox{as}\;\;(x,\lm)\to (\ox,\olm),
 \end{equation}
where $R$ is the residual function \eqref{res} and $(\ox,\olm)$ is a solution to the KKT system \eqref{vs}.
\begin{Theorem}[primal-dual linear convergence of ALM]\label{locpd} 
 Let $(\ox,\olm)$ be a solution to the KKT system \eqref{vs} and  let the SOSC \eqref{sosc} hold at $(\ox,\olm)$.
 Then there exist positive constants $\bar\gg$, $\bar\tau$, and $\bar\varrho$ such that for any starting point $(x^0,\lm^0)\in \B_{\bar\gg}(\ox,\olm)\cap\big(\Th\times \R^m\big)$
 there is a primal-dual sequence   $\{(\xk,\lk)\c$,  generated by Algorithm~{\rm\ref{Alg1}} with $\rok \ge \bar\varrho$ and   $\ek=\sigma(\xk,\lk)$ for all $k$, where $\sigma$ is defined by \eqref{error}, satisfying the  estimate
 \begin{equation}\label{cone}
 \|x^{k+1}-x^k\|+\|\lm^{k+1}-\lk\|\le \bar\tau R(x^k,\lk),
 \end{equation}
 where $R$ is the residual function defined by  \eqref{res}.
  Moreover, every such a sequence is convergent to $(\ox,\hat\lambda)$ for some $\hat\lambda\in\Lm(\ox)$, and its rate of convergence is linear.
  Furthermore, if $\rok\to \infty$,  the rate of convergence  of $\{(\xk,\lk)\c$ is superlinear. 
\end{Theorem}
\begin{proof}
 Define $R_k:=R(\xk,\lk)$   and observe that if $R_k=0$ for some $k$, then the pair $(\xk,\lk)$ satisfies the KKT system \eqref{vs}, and thus Algorithm \ref{Alg1} should stop. From now on assume that $R_k>0$ for all $k\in\N$. 
  Pick the positive constants $\kappa_1$ and $\gg_1$ from Theorem~\ref{sooc}, the positive constants $\gamma_2$ and $\kappa_2$ from Proposition~\ref{erct}, the positive constants 
  $\gamma_3$ and $\kappa_3$ from Theorem~\ref{est}, and the positive constants $\hat\gg$ and $ \tau$ from Proposition~\ref{solv}. By \eqref{error}, we can find $\gamma_4>0$ such that
\begin{equation}\label{ek}
\sigma(x,\lm)\le\dfrac{1}{4\kappa_1\kappa_2}R(x,\lm)\quad\textrm{ whenever }\;(x,\lm)\in\B_{\gg_4}(\ox,\olm).
\end{equation}
Define further the  positive  constants  $\bar\tau:=\kappa_3$ and 
\begin{equation}\label{const}
\bar\gamma:=\dfrac{\gamma'}{1+2\sqrt{2}\bar\tau\kappa_2} \quad\mbox{with}\quad
\gamma':=\min\Big\{\frac{\hat\gg}{2\tau},\gamma_3,\gamma_4,\dfrac{\gamma_1}{1+\sqrt{2}\bar\tau\kappa_2},\dfrac{\gamma_2}{1+\sqrt{2}
\bar\tau\kappa_2}\Big\}.
\end{equation}
Moreover, pick $q\in (0,1) $ and set 
\begin{equation}\label{rhof}
\bar\varrho:=\max\big\{\bar\rho, 4\kappa_1\kappa_2\bar\tau, \frac{2\sqrt{2} \kappa_1\kappa_2^2\bar\tau^2}{q}\big\},
\end{equation}
where  $\bar\rho$ is taken from Theorem~\ref{ugrowth}. 
Arguing by induction, we aim at showing that for any starting point $(x^0,\lambda^0)\in\big(\Th\times\R^m\big)\cap\B_{\bar\gamma}(\ox,\olm)$ there exists a sequence $\{(\xk,\lk)\c\subset \Th \times\R^m$, generated by Algorithm~\ref{Alg1}
with $\rok \ge \bar\varrho$ and   $\ek=\sigma(\xk,\lk)$ for all $k$,
 such that for all $k=0,1,\ldots$ the relationships
\begin{equation}\label{c1}
(\xk,\lk)\in\B_{\gamma'}(\ox,\olm),
\end{equation}
\begin{equation}\label{c2}
\dist\big(-\nabla_x L(\xkk,\lkk), N_\Th(\xkk)\big)\le\ek\quad\textrm{ and }\quad\Vert\xkk-\xk\Vert+\Vert\lkk-\lk\Vert\le\bar\tau R_k
\end{equation}
hold. To achieve this goal, set $k=0$, pick $(x^0,\lambda^0)\in\big(\Th\times\R^m\big)\cap\B_{\bar\gamma}(\ox,\olm)$, and observe from  $\bar\gamma\le\gamma'$ that \eqref{c1} holds in this case. We also have $\lambda^0\in\B_{\hat\gg/2\tau}(\olm)$ and $\rho_0\ge \bar\varrho$ and then find by Proposition~\ref{solv} a primal iterate $x^1\in \inte \B_{\hat\gamma}(\ox)$ satisfying $-\nabla_x\L(x^1,\lambda^0,\rho_0)\in N_\Th(x^1)$. Define further 
the multiplier $\lm^1:=\nabla e_{1^{}/\rho_0}g\big(\Phi(x^1)+ \rho_0^{-1}\lambda^0\big)$ and observe that
\begin{equation*}
-\nabla_x L(x^1,\lm^1)=-\nabla_x\L(x^1,\lambda^0,\rho_0)\in N_\Th(x^1).
\end{equation*}
Moreover, employing  Theorem~\ref{est} confirms that 
\begin{equation*}
\Vert x^1-x^0\Vert+\Vert\lm^1-\lambda^0\Vert\le\bar\tau R_0.
\end{equation*}
Thus $(x^1,\lambda^1)$ is well defined and satisfies the claimed condition \eqref{c2} for $k=0$. Assume by induction that the iterates $(\xk,\lk)$, $k=0,1,\ldots,s+1$, are well defined  and that   both conditions \eqref{c1} and \eqref{c2} hold for $k=0,1,\ldots,s$. We verify now the existence of $(x^{s+2},\lambda^{s+2})$ satisfying \eqref{c1} and \eqref{c2} for $k=s+1$. To this end, we first show that $(x^{s+1},\lambda^{s+1})\in\B_{\gamma'}(\ox,\olm)$. 
To proceed, fix an integer   $k$ with $0\le k\le s$.
Since $(\xk,\lk)\in\B_{\gamma'}(\ox,\olm)$, it follows from \eqref{eb1} and \eqref{c1} that
\begin{eqnarray}\label{c3}
R_k&\le&\kappa_2\big(\Vert\xk-\ox\Vert+\dist(\lk,\Lambda(\ox))\big)\\
&\le&\kappa_2\big(\Vert\xk-\ox\Vert+\Vert\lk-\olm\Vert\big)\le\sqrt{2}\kappa_2\Vert(\xk,\lk)-(\ox,\olm)\Vert\le\sqrt{2}\kappa_2\gamma'.\nonumber
\end{eqnarray}
Combining this with \eqref{c1} and the second relation in \eqref{c2}, we conclude that
\begin{eqnarray}\label{c4}
\Vert(\xkk,\lkk)-(\ox,\olm)\Vert&\le&\Vert\xkk-\xk\Vert+\Vert\lkk-\lk\Vert+\Vert(\xk,\lk)-(\ox,\olm)\Vert\nonumber\\
&\le&\bar\tau R_k+\gamma'\le(\sqrt{2}\bar\tau\kappa_2+1)\gamma'\le\gamma_1,
\end{eqnarray}
which implies that $(\xkk,\lkk)\in\B_{\gamma_1}(\ox,\olm)$. Using \eqref{subr3}, we obtain
\begin{eqnarray}\label{c5}
&&\Vert\xkk-\ox\Vert+\dist\big(\lkk,\Lambda(\ox)\big)\nonumber\\
&\le&\kappa_1 R_{k+1} =\kappa_1\Big(\dist\big(-\nabla_x L(\xkk,\lkk), N_\Th(\xkk)\big) +\big\Vert\Phi(\xkk)-{\rm{prox}}_g\big(\Phi(\xkk)+\lkk\big)\big\Vert\Big)\nonumber\\
&\le&\kappa_1\ek+\kappa_1\big\Vert\Phi(\xkk)-{\rm{prox}}_g\big(\Phi(\xkk)+\lkk\big)\big\Vert.
\end{eqnarray}
To proceed further,  define
\begin{equation}\label{skk}
\skk:=\proxk\big(\Phi(\xkk)+ \rho_k^{-1}{\lk}\big ).
\end{equation}
It follows from  \eqref{lkk}  and \eqref{ueq19} that
\begin{equation}\label{eq3}
\Phi(\xkk)-\skk= \rho_k^{-1}(\lkk-\lk).
\end{equation}
Since $\lkk=\nabla\menvk\big(\Phi(\xkk)+ \rho_k^{-1}{\lk}\big)$, we deduce from the second equality in \eqref{prm} that 
\begin{equation*}\label{clq17}
\lkk   \in \sub g\big(\Phi(\xkk)+ \rho_k^{-1}(\lk-\lkk)\big)=\partial g(\skk).
\end{equation*} 
Appealing now to the first equality in \eqref{prm} yields $\skk={\rm{prox}}_g(\skk+\lkk)$. Since the mapping $y\mapsto y-{\rm{prox}}_g(y+\lkk)$ is nonexpansive (cf. \cite[Proposition~12.27]{bc}), we arrive at the relationships
\begin{eqnarray}\label{eq2}
&&\big\Vert\Phi(\xkk)-{\rm{prox}}_g\big(\Phi(\xkk)+\lkk\big)\big\Vert\nonumber\\
&=&\big\Vert\Phi(\xkk)-{\rm{prox}}_g\big(\Phi(\xkk)+\lkk\big)\big\Vert-\big\Vert\skk-{\rm{prox}}_g(\skk+\lkk)\big\Vert\nonumber\\
&\le&\big\Vert\Phi(\xkk)-{\rm{prox}}_g\big(\Phi(\xkk)+\lkk\big)-\big(\skk-{\rm{prox}}_g(\skk+\lkk)\big)\big\Vert\nonumber\\
&\le&\Vert\Phi(\xkk)-\skk\Vert.
\end{eqnarray}
Combining \eqref{c5}, \eqref{eq3}, and \eqref{eq2} brings us to
\begin{equation*}
\Vert\xkk-\ox\Vert+\dist\big(\lkk,\Lambda(\ox)\big)\le\kappa_1\ek+\rho_k^{-1}{\kappa_1}\Vert\lkk-\lk\Vert.
\end{equation*}
This,  together with \eqref{ek}, \eqref{rhof}, the second inequality in \eqref{c2}, and \eqref{c3} results in 
\begin{eqnarray}\label{c10}
\Vert \xkk-\ox\Vert+\dist\big(\lkk,\Lambda(\ox)\big)&\le&\kappa_1\ek+\dfrac{\bar\tau\kappa_1}{\rok}R_k\le\dfrac{1}{4\kappa_2}R_k +\dfrac{1}{4\kappa_2}R_k\\
&\le&\dfrac{1}{2}\Big(\Vert\xk-\ox\Vert+\dist\big(\lk,\Lambda(\ox)\big)\Big)\nonumber,
\end{eqnarray}
which in turn implies the primal-dual estimates
\begin{equation}\label{c6}
\Vert\xkk-\ox\Vert+\dist\big(\lkk,\Lambda(\ox)\big)\le\dfrac{1}{2^{k+1}}\Big(\Vert x^0-\ox\Vert+\dist\big(\lambda^0,\Lambda(\ox)\big)\Big).
\end{equation}
This estimate, \eqref{c3},  and the second inequality in \eqref{c2} tell us that 
\begin{eqnarray}\label{c7}
\Vert(x^{s+1},\lambda^{s+1})-(x^0,\lambda^0)\Vert&\le&\sum\limits_{k=0}^s\Vert(\xkk,\lkk)-(\xk,\lk)\Vert \le\bar\tau\sum\limits_{k=0}^s R_k\nonumber\\
&\le&\bar\tau\kappa_2\sum\limits_{k=0}^s\Vert\xk-\ox\Vert+\dist\big(\lk,\Lambda(\ox)\big)\nonumber\\
&\le&\bar\tau\kappa_2\sum\limits_{k=0}^s\dfrac{1}{2^k}\Big(\Vert x^0-\ox\Vert+\dist\big(\lambda^0,\Lambda(\ox)\big)\Big)\\
&\le&2\bar\tau\kappa_2\big(\Vert x^0-\ox\Vert+\dist(\lambda^0,\Lambda(\ox))\big)\nonumber\\
&\le&2\bar\tau\kappa_2\big(\Vert x^0-\ox\Vert+\Vert\lambda^0-\olm\Vert\big).\nonumber
\end{eqnarray}
Thus, we arrive at the estimates
\begin{eqnarray*}
\Vert(x^{s+1},\lambda^{s+1})-(\ox,\olm)\Vert&\le&\Vert(x^{s+1},\lambda^{s+1})-(x^0,\lambda^0)\Vert+\Vert(x^0,\lambda^0)-(\ox,\olm)\Vert\\
&\le&2\bar\tau\kappa_2\big(\Vert x^0-\ox\Vert+\Vert\lambda^0-\olm\Vert\big)+\Vert(x^0,\lambda^0)-(\ox,\olm)\Vert\\
&\le&(2\sqrt{2}\bar\tau\kappa_2+1)\Vert(x^0,\lambda^0)-(\ox,\olm)\Vert\le(2\sqrt{2}\bar\tau\kappa_2+1)\bar\gamma=\gamma',
\end{eqnarray*}
where the last inequality follows from $(x^0,\lambda^0)\in\B_{\bar\gamma}(\ox,\olm)$, while the last equality comes from the definition of $\bar\gamma$ in \eqref{const}. 
This verifies that $(x^{s+1},\lambda^{s+1})\in\B_{\gamma'}(\ox,\olm)$. By  \eqref{const}, we get  $\lambda^{s+1}\in\B_{\hat\gg/2\tau}(\olm)$, and hence Proposition~\ref{solv} ensures the existence of of an optimal solution $x^{s+2}$ 
to \eqref{regp} associated with $(\lm^{s+1},\rho_{s+1})$ such that $x^{s+2}\in \inte\B_{\hat\gamma}(\ox)$. The latter yields 
\begin{equation*}
-\nabla_x\L(x^{s+2},\lambda^{s+1},\rho_{s+1})\in N_\Th(x^{s+2}).
\end{equation*} 
Set now $\lm^{s+2}:=\nabla e_{1^{}/\rho_{s+1}}\big(\Phi(x^{s+2})+\rho_{s+1}^{-1}\lambda^{s+1}\big)$ and observe that
\begin{equation*}
 -\nabla_x L(x^{s+2},\lm^{s+2}) =  -\nabla_x\L(x^{s+2},\lambda^{s+1},\rho_{s+1}) \in N_\Th(x^{s+2}),
\end{equation*}
which by Proposition~\ref{est} brings us to the estimate
\begin{equation*}
\Vert x^{s+2}-x^{s+1}\Vert+\Vert\lm^{s+2}-\lambda^{s+1}\Vert\le\bar\tau R_{s+1}.
\end{equation*}
This tells us that the iterate $(x^{s+2},\lambda^{s+2})$ satisfies   \eqref{c2} for $k=s+1$, which completes  the proof of  induction. Thus there exists a sequence  $\{(\xk,\lk)\c\subset\Th\times\R^m$, generated by Algorithm~\ref{Alg1} 
with $\rok \ge \bar\varrho$ and   $\ek=o\big(R(\xk,\lk)\big)$ for all $k$, for which  \eqref{c1} and \eqref{c2} fulfill.

Next, we are going to justify  the convergence of this primal-dual   sequence. Using the same arguments as in the proofs of \eqref{c6} and \eqref{c7} leads us to
\begin{equation}\label{c8}
\Vert(x^{k+l},\lambda^{k+l})-(\xk,\lk)\Vert\le 2\bar\tau\kappa_2\big(\Vert\xk-\ox\Vert+\dist(\lk,\Lambda(\ox))\big)\quad\mbox{for all}\;\;k,l\in\N.
\end{equation}
It follows from \eqref{c6} that the right-hand side of \eqref{c8} goes to zero as $k\to\infty$, which tells us that the sequence $\{(\xk,\lk)\c$ is Cauchy. By \eqref{c6},   this sequence converges  to $(\ox,\hat\lambda)$ for some $\hat\lambda\in\Lambda(\ox)$.
Now, we show  that  the rate of this convergence is linear. Indeed, letting $l\to\infty$ in \eqref{c8} tells us that for all $k\in\N$ we have
\begin{equation*}\label{c9}
\Vert(\xk,\lk)-(\ox,\hat\lambda)\Vert\le 2\bar\tau\kappa_2\big(\Vert\xk-\ox\Vert+\dist(\lk,\Lambda(\ox))\big),
\end{equation*}
which together with the first inequality in \eqref{c10} verifies that
\begin{eqnarray*}
\Vert(\xkk,\lkk)-(\ox,\hat\lambda)\Vert&\le&2\bar\tau\kappa_2\big(\Vert\xkk-\ox\Vert+\dist(\lkk,\Lambda(\ox))\big)\nonumber\\
&\le&2\bar\tau\kappa_1\kappa_2\big(R_k^{-1}{\ek}+\rho_k^{-1}\bar\tau\big)R_k\nonumber\\
&\le&2\bar\tau\kappa_1\kappa_2^2\big(R_k^{-1}{\ek}+\rho_k^{-1}\bar\tau\big)\big(\Vert\xk-\ox\Vert+\dist(\lk,\Lambda(\ox))\big)\nonumber\\
&\le&2\bar\tau\kappa_1\kappa_2^2\big(R_k^{-1}{\ek}+\rho_k^{-1}\bar\tau\big)\big(\Vert\xk-\ox\Vert+\Vert\lk-\hat\lambda\Vert\big)\nonumber\\
&\le& 2\sqrt{2}\bar\tau\kappa_1\kappa_2^2\big(R_k^{-1}{\ek}+\rho_k^{-1}\bar\tau\big)\Vert(\xk,\lk)-(\ox,\hat\lambda)\Vert.
\end{eqnarray*}
Combining this with $\rok\ge \bar\varrho$  with $ \bar\varrho$ taken from \eqref{rhof} and using the choice of the tolerance $\ek=o(R_k)$ result in 
\begin{equation*}
\limsup_{k\to\infty}\dfrac{\Vert(\xkk,\lkk)-(\ox,\hat\lambda)\Vert}{\Vert(\xk,\lk)-(\ox,\hat\lambda)\Vert}\le\limsup_{k\to\infty}2\sqrt{2}\bar\tau
\kappa_1\kappa_2^2\big(R_k^{-1}{\ek}+\rho_k^{-1}\bar\tau\big)\le q.
\end{equation*}
It follows from  $q\in (0,1)$ that the rate of   convergence of $\{(\xk,\lk)\c$ is linear. Furthermore, it is clear from the above inequality that if $\rok\to \infty$, we 
arrive at the superlinear convergence of the latter primal-dual sequence, which completes the proof.
\end{proof} 

The above theorem  extends a similar result, established in \cite[Theorem~3.4]{fs12} for nonlinear programming problems 
under a similar assumption, to the framework of  the composite optimization problem \eqref{comp} and thus significantly broaden the scope of this 
 result for important classes of composite problems, covered by \eqref{comp}. 

We end this section with a discussion on the selection of the 
parameter $\rho$ in the ALM for the composite problem \eqref{comp}. It is well-known that the boundedness of the penalty parameters $\rok$
in the computational implementation of Algorithm~\ref{Alg1} can alleviate the hurdle of dealing with ill-conditioned subproblems of the ALM. 
To avoid such an issue,  we can update the penalty parameter $\rok$  at each iteration in a particular way instead of using a pre-constructed sequence. 
To achieve this goal,   define   the auxiliary function $V: \mathbb{R}^n\times \mathbb{R}^m\times (0,\infty)\to (0,\infty)$ by 
\begin{equation}\label{aux}
V(x, \lambda, \rho):=\dist\big( -\nabla_x\L(x, \lambda, \rho), N_\Th(x)\big)+   \Vert\Phi(x)-\prox\big (\Phi(x)+ \rho^{-1}{\lambda}\big )  \Vert. 
\end{equation}
Given constants $\alpha \in (0, 1)$ and $r>1$, the penalty parameter in Algorithm~\ref{Alg1} can be updated according to the following rule:
\begin{equation}\label{rho}
\rokk := \begin{cases}
\rok\quad&\textrm{if } k=0\textrm{ or } V(\xkk, \lk, \rok)\leq\alpha V(\xk, \l1k, \ro1k),\\
r\rok&\textrm{otherwise}.
\end{cases}
\end{equation} 
In the framework of Theorem~\ref{locpd}, we are going to  show that the sequence $\{\rok\c$ in Algorithm~\ref{Alg1}, generated by \eqref{rho}, remains bounded.
To this end, it follows from the definitions of $V$ in \eqref{aux} and  the residual function $R$ in \eqref{res}, and from  \eqref{lkk} and \eqref{eq2} 
that 
\begin{equation*}\label{eq12}
V(\xkk, \lk, \rok)=\dist \big(-\nabla_xL(\xkk, \lkk), N_\Th(\xkk)\big)+ \Vert\Phi(\xkk)-\skk \Vert\geq R(\xkk,\lkk),
\end{equation*} 
where $\skk$ comes from \eqref{skk}. Similarly, we can show that $V(\xk, \l1k, \ro1k)\geq R(\xk,\lk)$.  On the other hand, we conclude from \eqref{cone},  \eqref{c2}, and \eqref{eq3} that
\begin{eqnarray*} 
V(\xkk, \lk, \rok)&=&\dist\big (-\nabla_xL(\xkk, \lkk), N_\Th(\xkk)\big)+ \rho_k^{-1}{\Vert \lkk-\lk\Vert }\\
&\leq & \ek+ \rho_k^{-1}\bar\tau R(\xk,\lk).
\end{eqnarray*} 
Combining these tells us that 
$$
\frac{V(\xkk, \lk, \rok)}{V(\xk, \l1k, \ro1k)}\le \frac{\ve_k}{R(\xk,\lk)}+  \rho_k^{-1}\bar\tau.
$$
Remember from  Theorem~\ref{locpd} that $\ek=o\big(R(\xk,\lk)\big)$ for all $k$. So, there exist $k_0\in \N$ such that 
$$
 \frac{\ve_k}{R(\xk,\lk)}<\frac{\al}{2}\quad \mbox{for all}\quad k\ge k_0.
$$
Moreover, choosing  $\bar \rho$ in Algorithm~\ref{Alg1}  sufficiently large allows us to conclude that   $\rho_k^{-1}\bar\tau<\al/2$ for all $k \in \N$.
Combining these ensures us  that 
\begin{equation}\label{bov}
\frac{V(\xkk, \lk, \rok)}{V(\xk, \l1k, \ro1k)}\le \al\quad \mbox{for all}\quad k\ge k_0.
\end{equation} 
Using \eqref{bov} for $k=k_0$ and then appealing to \eqref{rho} yield $\rho_{k_0+1}=\rho_{k_0}$.
 Now, an argument via induction together with \eqref{bov} and  \eqref{rho}  tells us that  $\rho_{k_0+s} = \rho_{k_0}$ for all $s\in\N$. Thus the sequence $\{\rok\c$ 
remains bounded in Algorithm~\ref{Alg1} provided that $\bar \rho$ in Algorithm~\ref{Alg1}  is sufficiently large and that we utilize \eqref{rho} to update the penalty parameters in this algorithm.

\section{Examples of Composite Optimization Problems} \label{sect06}\sce

In this section, we are going to present some important classes of constrained and composite optimization problems 
that can be covered by the composite problem \eqref{comp}. We begin with a discussion about the stopping criterion \eqref{xkk} as well as the updating scheme \eqref{rho}
for the composite problem \eqref{comp}.  According to \eqref{Th}, we always have $N_{\Th}(x)=\rge B^*$ for all $x\in \Th$. Moreover, Theorem~\ref{locpd} and \eqref{error}
provide us   a constructive way to choose the parameter $\ve_k$ in the stopping criterion \eqref{xkk}. To do this, choose a constant $\al>1$ and then set   
$$
\ve_k:=\big(R(\xk,\lk)\big)^\al,
$$
where the residual function $R$ is defined by \eqref{res}. Thus, \eqref{est0} can be equivalently expressed by 
\begin{equation*} 
\dist\big(-\nabla_x\L(\xkk, \lk, \rok), \rge B^*\big)\leq \big(R(\xk,\lk)\big)^\al.
\end{equation*} 
To compute $\nabla_x\L(\xkk, \lk, \rok)$, note that 
\begin{eqnarray*}
\nabla_x\L(\xkk, \lk, \rok)=\nabla \ph(\xkk)+ \nabla \Phi(\xkk)^*\nabla(e_{ {1}/{\rok}} g)(\Phi(\xkk)+\rok^{-1}\lk).
\end{eqnarray*}
Moreover, we know from \cite[Theorem~2.26]{rw} that $\nabla\menv(z) = \rho\big(z-\prox(z)\big)$ for any $z\in \R^m$. Combining these tells us that if 
we can compute the proximal mapping of $g$, then $\nabla_x\L(\xkk, \lk, \rok)$ can be calculated at each iteration of the ALM. Since we have 
$$
R(\xk,\lk)=\dist\big(-\nabla_x L(\xk, \lk), \rge B^*\big)+\big\|\Phi(\xk)-{\rm{prox}}_g\big(\Phi(\xk)+\lk\big)\big\|,
$$
the computation of the residual function $R$ at each iteration depends on our ability to calculate the proximal mapping of $g$ and the distance function of $\rge B^*$. 
In summary, to check the stopping criterion \eqref{xkk}, the updating scheme \eqref{rho}, and the residual function $R$ at each iteration 
of the ALM, one requires to compute 1) the proximal mapping of the CPLQ function $g$; 2) the distance function of $\rge B^*$.
Below, we assume that the constraint set $\Th$ is simple enough so that the distance function of the subspace $\rge B^*$ can be computed and thus consider 
some instances of the composite problem \eqref{comp} for which the proximal mapping of $g$ can be calculated  as well:

\begin{enumerate}
\item   $g=\dd_C$,  where $C=\{0\}^{s}\times \R^{m-s}_-$ with $0\le s\le m$, and $\Th=\R^n$. In this case the composite problem \eqref{comp}
reduces to a nonlinear programming problem with $s$ equality constraints and $m-s$ inequality constraints  for which the ALM was studied in \cite{Ber82, fs12, iks15}. In this setting, the proximal mapping of $g$ amounts to the projection mapping $P_C$,
which can be computed by \cite[Lemma~6.26]{ab}.
Note that the SOSC \eqref{sosc} boils down to the classical second-order sufficient condition for nonlinear programs. In this framework, 
Theorem~\ref{locpd} covers \cite[Theorem~3.4]{fs12} in which the linear convergence of the ALM under the SOSC was established. 

\item $g(z)=\max\{z_1,\ldots,z_m\}$, where $z=(z_1,\ldots,z_m)$, and $\Th=\R^n$. In this case, the composite problem \eqref{comp} reduces to the unconstrained  minimax problem
$$
\mbox{minimize}\;\;\ph(x)+\max\big\{\ph_1(x), \ldots, \ph_m(x)\big\}\quad \mbox{subject to}\;\; x\in \R^n,
$$ 
where $\Phi=(\ph_1,\ldots,\ph_m)$ in \eqref{comp} with $\ph_i:\R^n\to \R$ for $i=1,\ldots,m$.
The proximal mapping of $g$ can be obtained via  \cite[Example~6.49]{ab} as
$$
\prox(z) =z-\rho^{-1}P_{\Delta_m}(\rho z),
$$
where $\Delta_m$ stands for the unit simplex in $\R^m$. Note also the projection mapping of the unit simplex can be fully computed; see \cite[Corollary~6.29]{ab}. 
It is worth mentioning that the SOSC \eqref{sosc} can be simplified via \cite[Example~13.16]{rw} as 
\begin{equation}\label{rev1}
\langle\nabla_{xx}^2L(\bar x,\olm)w,w\rangle>0\quad \mbox{for all }\; w\neq 0\;\;\mbox{with}\;\; \nabla \Phi(\ox)w\in K_g \big(\Phi(\ox),\olm \big),
\end{equation}
since $\d^2g (\Phi(\bar x),\olm) (u)=0$ for all $u\in K_g \big(\Phi(\ox),\olm \big)$.
\item $g$ is the   $\|\cdot \|_\infty$ or $\|\cdot\|_1$ in $\R^m$ with $\Th=\R^n$.  In the case, the composite problem \eqref{comp} covers the regularized linear and nonlinear least square  optimization problems with 
$g$ having a role in {\em regularizing} optimal solutions. The proximal mapping of $g$ for these cases can be found in \cite[Example~6.8]{ab} and \cite[Example~6.48]{ab}, respectively. 
Similar to (b), one can observe via \cite[Proposition~13.9]{rw} that the SOSC \eqref{sosc} can be  equivalently expressed  as \eqref{rev1} since  $\d^2g (\Phi(\bar x),\olm) (u)=0$ for all $u\in K_g \big(\Phi(\ox),\olm \big)$.

\item $g$ is a CPLQ function,  defined by
\begin{equation*}
g(z) := \sup_{y\in C}\Big\{\langle z, y\rangle - \tfrac{1}{2}\langle y, B y\rangle\Big\},\;\;z\in \R^m,
\end{equation*}
where $C$ is a  polyhedral convex set $\R^m$, and $B$ is an $m\times m$  symmetric and positive semidefinite matrix . 
If , in addition, we assume that the objective function $\ph$ in \eqref{comp} has a quadratic representation 
$$
\ph(x)=\la q,x\ra+\sm \la Qx,x\ra,
$$
where $q\in \R^n$, and  $Q$ is an $n\times n$ symmetric matrix, then the composite problem \eqref{comp}
falls into the class of extended linear-quadratic programming problems, which goes back to Rockafellar and Wets \cite{rw86}.
To find the proximal mapping of $g$ for this setting, observe that  (cf. \cite[Theorem~6.45]{ab})
$$
\prox(z)+ \rho^{-1}\textrm{prox}_{\rho g^*}(\rho z)=z, \;\;z\in \R^m, 
$$
where $g^*$ stands for the Fenchel conjugate of $g$ in the sense of convex analysis. Moreover, it is easy to see that 
$g^* =  j_B+\delta_C$ with $j_B(y): = \tfrac{1}{2}\langle y, B y\rangle$  for any $y\in \R^m$, which  leads us to 
\begin{equation*}
\textrm{prox}_{\rho g^*}(\rho z) = \argmin_{y\in C}\Big\{j_B(y)+\tfrac{1}{2\rho}\|y-\rho z\|^2\Big\}
=\Big\{y\in C\, \big|\, z\in (B+\rho^{-1}I)y+N_C(y)\Big\}.
\end{equation*}
Thus, for a given point $z\in \R^m$, the computation of $\textrm{prox}_{\rho g^*}(\rho z)$ 
reduces to solving a strongly monotone variational inequality, which can be achieved by the projected gradient descent method, see, e.g., \cite[Algorithm~3.4]{Noo91} and \cite[Algorithm~12.1.1]{FaP03}.

\item  $g$ could be a mixture of different possibilities, discussed in (a)-(d).   This means that the vector $z$ can be partitioned into $z^j$ for $j=1,\ldots,s$ with $s\in \N$ and 
the corresponding modeling function $g_j$ is  chosen from those listed in (a)-(d) and then $g(z^1,\ldots,z^s)=\sum_{i=1}^s g_i(z^i)$. According to \cite[Theorem~6.6]{ab}
we have 
$$
\mbox{prox}_g(z^1,\ldots,z^s)=\mbox{prox}_{g_1}(z^1)\times \cdots\times \mbox{prox}_{g_s}(z^s).
$$
As argued above, for each $j=1,\ldots,s$, the proximal mapping  $\mbox{prox}_{g_j}(z^j)$ can be computed. This  allows us to  find $\mbox{prox}_g(z^1,\ldots,z^s)$ for this setting.
\end{enumerate}

{\bf Acknowledgements.} We thank the two anonymous reviewers whose comments and suggestions helped improve the original presentation of the paper.


\end{document}